\documentclass[reqno]{amsart}
\usepackage{amsmath,amssymb}
\usepackage{graphics,epsfig,color}
\usepackage[mathscr]{euscript}
\newtheorem{theorem}{Theorem}[section]
\newtheorem{proposition}[theorem]{Proposition}
\newtheorem{lemma}[theorem]{Lemma}
\newtheorem{corollary}[theorem]{Corollary}
\theoremstyle{definition}
\newtheorem{definition}[theorem]{Definition}
\newtheorem{assumption}[theorem]{Assumption}

\numberwithin{equation}{section}
\numberwithin{theorem}{section}
\renewcommand{\epsilon}{\varepsilon}

\newcommand{\mc}[1]{{\mathcal #1}}
\newcommand{\bb}[1]{{\mathbb #1}}
\newcommand{\ms}[1]{{\mathscr #1}}
\newcommand{\bs}[1]{{\boldsymbol #1}}
\newcommand{\R}{\mathbb{R}}

\newcommand{\id}{{1 \mskip -5mu {\rm I}}}

\newcommand{\de}{\mathop{}\!\mathrm{d}}

\title{Large deviations  for Kac-like walks}
\author[G.\ Basile]{Giada Basile}
\address{Giada Basile \hfill\break \indent
   Dipartimento di Matematica, Universit\`a di Roma `La Sapienza'
   \hfill\break \indent
   P.le Aldo Moro 2, 00185 Roma, Italy}
 \email{basile@mat.uniroma1.it}
 \author[D.\ Benedetto]{Dario Benedetto}
 \address{Dario Benedetto \hfill\break \indent
   Dipartimento di Matematica, Universit\`a di Roma `La Sapienza'
   \hfill\break \indent
   P.le Aldo Moro 2, 00185 Roma, Italy}
 \email{benedetto@mat.uniroma1.it}
\author[L.\ Bertini]{Lorenzo Bertini}
\address{Lorenzo Bertini \hfill\break \indent
   Dipartimento di Matematica, Universit\`a di Roma `La Sapienza'
   \hfill\break \indent
   P.le Aldo Moro 2, 00185 Roma, Italy}
 \email{bertini@mat.uniroma1.it}
  \author[C.\ Orrieri]{Carlo Orrieri}
 \address{Carlo Orrieri \hfill\break \indent
	 Dipartimento di Matematica, Universit\`a di Pavia
	\hfill\break \indent
	Via Ferrata 5, 27100 Pavia (PV), Italy   
    }
 \email{carlo.orrieri@unipv.it}

\begin{document}
\begin{abstract}
  We introduce a Kac's type walk whose  rate  of binary collisions  preserves the total momentum but not the kinetic energy.
  In the limit of large number of particles we describe the dynamics in terms of empirical measure and flow, proving the corresponding large deviation principle.
  The associated  rate function has an explicit expression. As a byproduct of this analysis, we provide a gradient flow formulation of the Boltzmann-Kac equation. 
\end{abstract}


\maketitle
\thispagestyle{empty}

\section{Introduction}
\label{s:0}
The statistics of rarefied gas is described, at the kinetic level, by the Boltzmann equation.
It has become paradigmatic since
it encodes most of the conceptual and technical issues in the description of the
statistical properties for out of equilibrium systems. In the spatially homogeneous case the Boltzmann equation reads
\begin{equation}\label{eq:Bgen}\begin{split}
    \partial_t f_t(v)\de v
    = &
               \iint \!
               r(v', v'_*;\de v, \de v_*) f_t(v') f_t(v_*') \de v' \de v'_*\\
               &-
      f_t(v) \de v\iint\!\! r(v,v_*;\de v', \de v'_*)   f_t(v_*)\de v_*,
  \end{split}\end{equation}  
where $f(v)\de v$ is the one-particle velocity  distribution,
$v, v_*$ (resp.$v', v'_*$) are the incoming (resp.  outgoing) velocities in the binary collision and $r$ is the collision rate.
In the classical  case the collision rate is concentrated on the set of velocities satisfying the constrains of momentum and energy conservation
$v+v_*= v' + v'_*$ and  $|v|^2+|v_*|^2= |v'|^2 + |v'_*|^2$, and it satisfied the detailed balance condition
 $M(\de v) M(\de v_*) r(v,v_*;\de v',\de v_*')=M(\de v') M(\de v_*')r(v',v_*';\de v,\de v_*)$, where $M(\de v)$ is the Maxwellian.

In the pioneering work \cite{K}, Kac derived \eqref{eq:Bgen} from a stochastic model of $N$ particles, interacting via binary collisions satisfying the conservation of the energy.
This result can be seen as the law of large number for the empirical measure of the microscopic dynamics.
In particular, Kac's result established the validity of the \emph{Stosszahlansatz}.
We refer to \cite{Sn} for further developments and references.
The analysis of the corresponding fluctuations, in the central limit regime, has been carried out in \cite{Ta, Uc1, Uc2}.
Regarding the large deviation asymptotics, we point out that while the law of large numbers depends on the validity of the  \emph{Stosszahlansatz} with probability converging to one as $N\to \infty$, the large deviation principle requires that it holds with probability super-exponentially close to one for $N$ large.
It is  therefore a non trivial improvement of the Kac's result.
The first statement in this direction has been obtained in \cite{Le}, where a large deviation upper bound is proven. 
In \cite{Re}  a large deviation result  has been derived   for a stochastic model in the setting of one dimensional spatially dependent  Boltzmann equation with discrete  velocities.
A main   issue behind the proof of large deviation lower bound is to establish a law of large number for a perturbed dynamics, proving in particular the uniqueness of the perturbed Boltzmann equation,
which in general fails.

Regarding the Newtonian dynamics, in view of the previous discussion, the validity of a large deviation principle is a most challenging issue. 
The general structure of the rate function associated to the Boltzmann equation for hard sphere is discussed in \cite{Bou}.
A derivation from Newtonian dynamics is presented in \cite{BGSS2}, see also \cite{BGSS1} for a comparison of the rate function derived in \cite{BGSS2} with the one proposed in \cite{Bou}.  

We focus on the  large deviation principle  for  Kac-type spatially homogeneous models. Beside the empirical density, it is convenient to introduce another observable, the empirical flow, which records the incoming and outgoing velocities in the collisions, i.e., letting $N$ the number of particles,
$N q_t(\de v,\de v_*, \de v', \de v'_*)\de t$ is the number of collisions in the time-window $[t, t+\de t]$ with incoming velocities in $\de v\,\de v_*$ and outgoing velocities in $\de v'\,\de v'_*$. 
In particular, the mass of the empirical flow  is the normalized  total number of collisions.
The empirical measure and flow are linked by the balance equation
\begin{equation}\label{cer}
  \begin{split}
\partial_t f_t(v) \de v = \int \!
               \Big[q_t(\de v', \de v'_*;\de v, \de v_*) + q_t(\de v', \de v'_*;\de v_*, \de v) \\
                 - q_t(\de v,\de v_*;\de v', \de v'_*)-q_t(\de v_*,\de v;\de v', \de v'_*)\Big],
              \end{split} \end{equation}
which expresses the conservation   of the mass. The idea of considering the pair of observables empirical measure and flow has been exploited in the context of Markov processes in
\cite{BFG, BB, GR}.
In this setting  the rate function relative to the pair empirical measure and flow has a closed form
and it is equal to  $I(f, q)=I_0(f_0) + J(f,q)$, where $I_0(f_0)$ takes into account the fluctuations of  the initial datum and the dynamical term $J$ is given by
\begin{equation}\label{Jr}
  J(f,q)= \int_0^T \!\de t\int \Big\{ \de q_t \log \frac{\de q_t}{\de q^{f_t}} -
  \de q_t + \de q^{f_t}\Big\},
\end{equation}  
where $q^f(\de v, \de v_*, \de v', \de v'_*)= \frac 1 2 f(v)f(v_*)\de v\,\de v_* r(v, v_*; \de v', \de v'_*) $. By projecting $J$ on the empirical density $f(v)dv$ one recovers the variational expression for the rate function  obtained for  the empirical measure in \cite{Re, Le, Bou}.

The rate  function \eqref{Jr} has a simple interpretation in terms of Poisson point processes.  Let $\{Z^N\}$  be a sequence of  Poisson random variables with parameters $N\lambda$. By Stirling's formula,
for $N$ large
$$
\bb P (Z^N =N q)\approx \exp\{-N[q\log(q/\lambda) -q+\lambda]\}.
$$
Therefore, for $N$ large the statistics of the collisions in the Kac's walk can be thought as sampled from a Poisson point process with intensity $N q^f$, where now $f$ and $q$ are related by the balance equation \eqref{cer}.

Here we implement this program for  a  model of $N$ particles, interacting via stochastic binary collisions satisfying the conservation of the momentum, but not of the kinetic energy.
Such a model is relevant, as example, in the case  of a   molecular gas  when the internal degrees of freedom are disregarded.
We do not assume a detailed balance condition and the corresponding Boltzmann-Kac equation is of the form \eqref{eq:Bgen}.
We prove the large deviation upper bound with the rate function introduced above.
The proof of the matching lower bound is achieved when $q$ has bounded second moment. From a technical viewpoint, the advantage of momentum conservation with respect to energy conservation is the linearity of the constraint, which allows to use convolution in the approximation argument for the lower bound.
We also derive the variational formula for  the projection of the rate function  on the empirical measure. 

In the context of i.i.d.\ Brownians, the connection between the large deviation rate function and the gradient flow formulation of the heat equation is discussed in \cite{ADPZ}, see also \cite{BBB, MPR} for the case of i.i.d.\ reversible  Markov chains. Here  we derive a gradient flow formulation for the Boltzmann-Kac equation from the large deviation rate function \eqref{Jr}. On general grounds, a gradient flow formulation of evolution equations is based on the choice of a pair of functions in Legendre duality. In \cite{Er2} it is shown how this pair can be chosen so that the Boltzmann-Kac equation is the gradient flow of the entropy with respect to a suitable distance on the set of probability measures.
As we here show, the choice in \cite{Er2} is not the one associated to the large deviation rate function. Instead, analogously to  \cite{BBB}, in  the formulation here presented
the non-linear Dirichlet form associated to the Boltzmann-Kac equation plays the role of 
the slope of the entropy.

\section{Notation and main result}
\label{s:1}

\subsection*{Kac walk}
For a Polish space $X$ we denote by $\mc M(X)$ the set of positive
Radon measures on $X$ with finite mass; we consider $\mc M(X)$
endowed with the weak* topology and the associated Borel
$\sigma$-algebra.

Given $N\ge 2$, a \emph{configuration} is defined by $N$ velocities in
$\bb R^d$. The configuration space is therefore given by
$\Sigma_N:=(\bb R^d)^N$. Elements of $\Sigma_N$ are denoted by $\bs v
=(v_k)_{k=1,\ldots,N}$, with $v_k\in \bb R^d$. 
The Kac walk that we here consider is the
Markov processes on $\Sigma_N$ whose generator acts on continuous and
bounded functions $f\colon \Sigma_N\to \bb R$ as 
\begin{equation*}
  \ms L_N f = \frac 1N \sum_{\{i,j\}} L_{ij} f
\end{equation*}
where the sum is carried over the unordered pairs $\{i,j\}\subset \{1,..,N\}$, $i\neq j$, and 
\begin{equation*}
  L_{ij}f \, (\bs v) = \int r(v_i,v_j;dv',dv'_*)
  \big[ f(T_{ij}^{v',v'_*} \bs v) - f (\bs v) \big].
\end{equation*}
Here
\begin{equation*}
  \big(T_{ij}^{v',v'_*} \bs v\big)_k :=
  \begin{cases}
    v' & \textrm{if } k=i \\
    v_*'& \textrm{if } k=j \\
    v_k & \textrm{otherwise.} 
  \end{cases}
\end{equation*}
and  the \emph{collision rate} $r$ is a continuous function from $\bb
R^d\times \bb R^d$ to $\mc M(\bb R^d\times \bb R^d)$.
Finally,   the  \emph{scattering rate}  $\lambda\colon \bb R^d \times \bb R^d \to \bb R_+$ is
  \begin{equation}\label{sr}
    \lambda(v, v_*)=\int r(v, v_*; \de v', \de v'_*).
  \end{equation}


  We assume that $r$ satisfies the following conditions in which we set
   $\ms V :=\{v+v_*=v'+v_*'\}\subset (\bb R^d)^2\times (\bb R^d)^2$.
\begin{assumption}\label{ass:1}${}$
\begin{itemize}
\item [(i)]\emph{Conservation of momentum.} For each
  $(v,v_*) \in \bb R^d\times \bb R^d$ the measure $r(v,v_*;\cdot)$ is
  supported on the hyperplane
  $\{(v',v_*') \in \bb R^d \times \bb R^d \colon  v'+v_*' = v+v_*\}$.
\item [(ii)]\emph{Collisional symmetry.} For each
  $(v,v_*) \in \bb R^d\times \bb R^d$ the scattering kernel satisfies
  $r(v,v_*;\de v',\de v_*')=r(v_*, v; \de v', \de v'_*)=r(v,v_*;\de v'_*,\de v')$.
\item [(iii)]\emph{Non degeneracy of  the scattering kernel.}
  There exists a density $B \colon\ms V\to \bb R_+$
  such that $r(v, v_*, \de v', \de v'_*)= B(v, v_*, w') \de w'$, where
  $w'=(v' -v'_*)/\sqrt 2$.
  Moreover, there exists $c_0>0$ such that  $B(v, v_*, w')\geq c_0\exp\{- c_0 \, |w'|^2\}$. 
\item [(iv)] \emph{Gaussian tails.} There exists $C<+\infty$, $\eta >0$, $\gamma\in [0,2)$ such that for any $v$, $v_* \in \bb R^d$
  \begin{equation*}
    \int\!dw'\, B(v,v_*,w') e^{\eta  |w'|^2} \leq  C (1+ |v-v_*|^\gamma).
    \end{equation*} 
  \item[(v)]\emph{Point-wise bound}. There exists a constant $C<+\infty$ such that $ B(v,v_*,w')\leq Ce^{C(|v|^2+|v_*|^2 + |w'|^2)}$.
\end{itemize}
\end{assumption}
We remark that we do not assume balance conditions. Observe that item (iv) implies
\begin{equation}\label{(v)}
\lambda(v,v_*)\leq C(1 + |v-v_*|^\gamma).
\end{equation}  
An example of a  scattering kernel meeting the above conditions is $B(v, v_*, w')= (1 + |v-v_*|) e^{-|w'|^2}$.

We denote by $(\bs v(t))_{t\ge 0}$ the Markov process generated by
$\ms L_N$. Let
$\Sigma_{N,0} := \big\{ \bs v \in \Sigma_N \colon N^{-1}\sum_{k} v_k=0 \big\}$
be the subset of configurations with zero average velocity, that  it is invariant by the dynamics in view of the conservation of the momentum.
By the positivity of the collision rate (see Assumption \ref{ass:1}, item (iii)),
the Kac walk  is ergodic when restricted to $\Sigma_{N,0}$.
We shall consider the Kac's walk restricted to $\Sigma_{N,0}$.

Fix hereafter $T>0$. Given a probability $\nu$ on $\Sigma_{N,0}$ we denote
by $\bb P_\nu^N$ the law of the Kac walk on the time interval $[0,T]$.
Observe that $\bb P_\nu^N$ is a probability on the Skorokhod
space $D([0,T];\Sigma_{N,0})$.
As usual if $\nu=\delta_{\bs v}$ for some $\bs v \in \Sigma_{N,0}$, the
corresponding law is simply denoted by $\bb P_{\bs v}^N$.

\subsection*{Empirical measure and flow}
We denote by $\ms P_0(\bb R^d)$ the set of probability measures on $\bb R^d$
with zero mean.
We consider $\ms P_0(\bb R^d)$ as a closed subset of the space of probability measure with finite mean equipped with the $W_1$ Wasserstein distance. Then  $\ms P_0(\bb R^d)$ endowed with the relative  topology is a Polish space.
The \emph{empirical measure} is the map
$\pi^N\colon \Sigma_{N,0} \to \ms P_0(\bb R^d)$ defined by
\begin{equation}
  \label{1}
  \pi^N(\bs v) :=\frac 1N \sum_{i=1}^N\delta_{v_i}.
\end{equation}

Let $D\big([0,T]; \ms P_0(\bb R^d)\big)$ the set of
$\ms P_0(\bb R^d)$-valued c{\'a}dl{\'a}g paths endowed with the
Skorokhod topology and the corresponding Borel $\sigma$-algebra.
With a slight abuse of notation we denote also by $\pi^N$ the map 
from $D\big([0,T]; \Sigma_{N,0} \big)$ to $D\big([0,T]; \ms P_0(\bb
R^d)\big)$ defined by $\pi^N_t(\bs v(\cdot)):=\pi^N(\bs v(t))$, $t\in
[0,T]$. 

We denote by $\ms M$ the (closed) subset of
$\mc M([0,T]\times (\bb R^d)^2\times (\bb R^d)^2\big)$ given by the
measures $Q$ that satisfy
$$Q(\de t; \de v,\de v_*, \de v',\de v_*')=Q(\de t; \de v_*,\de v, \de v',\de v_*')=Q(\de t; \de v,\de v_*, \de v'_*,\de v').$$
The \emph{empirical flow} is the map $Q^N\colon D\big([0,T]; \Sigma_{N,0} \big) \to \ms M$
defined by
\begin{equation}
  \label{2}
  Q^N(\bs v) (F) :=\frac 1N
  \sum_{\{i,j\}} \sum_{k\ge 1} F\big(\tau^{i,j}_k;
  v_i(\tau^{i,j}_k-),v_j({\tau^{i,j}_k}-),
  v_i(\tau^{i,j}_k),v_j(\tau^{i,j}_k)\big) 
  \quad 
\end{equation}
where $F\colon [0,T]\times (\bb R^d)^2\times (\bb R^d)^2\to \bb R$ satisfies $F(t;v,v_*,v',v'_*) =F(t;v_*,v,v',v'_*) = F(t;v,v_*,v'_*,v')$, and $F$ is continuous and bounded, 
while $(\tau^{i,j}_k)_{k\ge 1}$ are the
jump times of the pair $(v_i,v_j)$. Finally, $v_i(t-) = \lim_{s\uparrow t} v_i(s)$. 

For each $\bs v \in \Sigma_{N,0}$ with $\bb P^N_{\bs v}$ probability
one the pair $(\pi^N,Q^N)$ satisfies the following
balance equation that express the conservation of probability. For each
$\phi\colon [0,T]\times \bb R^d\to \bb R$ bounded, continuous, and
continuously differentiable with respect to time 
\begin{equation}
  \label{bal1}
  \begin{split}
  &\pi^N_T(\phi_T)-\pi^N_0(\phi_0)-\int_0^T\! \de t\, 
  \pi^N_t(\partial_t \phi_t)\\
  &\qquad 
  +\int Q^N(\de t;\de v,\de v_*,\de v',\de v_*')\big[ \phi_t(v)+\phi_t(v_*)
  -\phi_t(v')-\phi_t(v_*')\big] =0.
  \end{split}
\end{equation}
In view of the conservation of the momentum, the
measure $Q^N(\de t;\cdot)$ is supported on the hyperplane
$\ms V $.

\subsection*{The rate function}
Let $\ms S$ be the (closed) subset of
$D\big([0,T]; \ms P_0(\bb R^d)\big)\times \ms M$ given by elements
$(\pi,Q)$ that satisfies the balance equation
\begin{equation}
  \label{bal}
  \begin{split}
  &\pi_T(\phi_T)-\pi_0(\phi_0)-\int_0^T\! \de t\, 
  \pi_t(\partial_t \phi_t)\\
  &\qquad 
  +\int Q(\de t;\de v,\de v_*,\de v',\de v_*')\big[ \phi_t(v)+\phi_t(v_*)
  -\phi_t(v')-\phi_t(v_*')\big] =0
  \end{split}
\end{equation}
for each $\phi:[0,T]\times \bb R^d\to \bb R$ continuous, bounded and continuously differentiable in $t$, with bounded derivative.
We consider $\ms S$ endowed with the relative topology and the corresponding Borel $\sigma$-algebra.

For $\pi\in D\big([0,T]; \ms P_0(\bb R^d)\big)$ let $Q^\pi$ be
the measure defined by
\begin{equation}
  \label{4}
  Q^\pi(\de t;\de v,\de v_*,\de v',\de v_*') := \frac 1 2 dt \, \pi_t(\de v) \pi_t(\de v_*)
  \, r(v,v_*;\de v',\de v_*')
\end{equation}
and observe that $Q^\pi(\de t,\cdot)$ is supported on $\ms V$.

\begin{definition}
  \label{def:sac}
  Let $\ms S_\mathrm{ac}$ be the subset of $\ms S$ given by the elements
  $(\pi,Q)$ that satisfy the following conditions:
  \begin{itemize}
  \item [(i)] $\pi\in C\big([0,T];\ms P_0(\bb R^d)\big)$;
  \item [(ii)] $\sup_{t\in [0,T]} \pi_t (\zeta) <+\infty$, where
    $\zeta(v) = |v|^2$;
  \item [(iii)] $Q\ll Q^\pi$.
  \end{itemize}
\end{definition}

Observe that by item (iv) of Assumption \ref{ass:1}, condition (ii)
implies that if $(\pi,Q)\in \ms S_\mathrm{ac}$ then $Q^\pi$ is a
finite measure.
Moreover, by choosing positive functions $\phi$ not depending on $t$ in the balance
equation \eqref{bal} and neglecting the loss term we obtain
$$
\pi_t(\phi)\leq \pi_0(\phi) +2\int_0^t\int Q(\de s, \de v, \de v_*, \de v', \de v'_*) \phi(v'). 
$$
Since $Q\ll Q^\pi$ and, by Assumption \ref{ass:1}, item (iii), the  marginal on $v'$ of $Q^\pi$ is absolutely continuous with respect to the
Lebesgue measure, 
we deduce that
$\pi_0\ll \de v$ implies $\pi_t\ll \de v$, for any $t\geq 0$. As a
consequence, also $Q$ is absolutely continuous with respect to the
Lebesgue measure on $[0,T]\times \ms V$.

The dynamical rate function $J\colon \ms S \to [0,+\infty]$ is defined by
\begin{equation}
  \label{5}
  J(\pi,Q):=
  \begin{cases}
    {\displaystyle 
    \int  \de Q^\pi \Big[ 
    \, \frac{\de Q\phantom{^\pi}}{\de Q^\pi} \log \frac{\de Q\phantom{^\pi}}{\de Q^\pi} -
    \Big( \frac{\de Q\phantom{^\pi}}{\de Q^\pi}  -1\Big)\Big]  } & \textrm{if } (\pi,Q)\in \ms
    S_\mathrm{ac}\\ \\
    + \infty& \textrm{otherwise. } 
  \end{cases}
\end{equation}

In order to obtain a large deviation principle, chaotic initial conditions are not sufficient but we need that the empirical measure at time zero satisfies a large deviation principle.
Referring  to \cite{CCL-RLV} for a discussion on \emph{entropically chaotic} initial conditions, we next provide an example of a class of allowed initial data.

\begin{assumption}\label{ass:2}
  Given  $m\in \ms P_0(\bb R^d)$ set $\mu^N=m^{\otimes N}$ and choose as initial distribution of the Kac's walk the probability on $\Sigma_{N,0}$ given by
$\nu^N = \mu^N (\  \cdot\  \vert  \sum_{i} v_i = 0)$. We assume that 
$m$ is absolutely continuous with respect to
the Lebesgue measure and still denote by $m$ its density.
We further more  assume that there exists $\gamma>0$ such that
\begin{itemize}
  \item[(i)] $\int \de v\, m(v)
\exp\{\gamma |v|^2\} < +\infty$; 
\item[(ii)]
the Fourier transform of $m(v)\exp\{\gamma |v|^2\}$
is in $L^1(\bb R^d)$;
\item[(iii)]$m(v)\geq \gamma \exp\{-\frac 1 \gamma |v|^2\}$.
\end{itemize}
\end{assumption}

Given two probabilities $\mu_1, \mu_2 \in\ms P_0(\bb R^d)$, the relative entropy
$H(\mu_2\vert\mu_1)$ is defined as $H(\mu_2\vert\mu_1)=\int \de \mu_1 \rho\log\rho$, where $\de \mu_2=\rho\, \de \mu_1$, understanding that $H(\mu_2|\mu_1)=+\infty$ if $\mu_2$ is not absolutely continuous with respect to $\mu_1$.

Denoting by $\hat{\ms S}$  the set of paths $(\pi, Q)\in \ms S$ such that
  \begin{equation}\label{QV}
\displaystyle \int Q(\de t; \de v,\de v_*,\de v',\de v'_*)|v+v_*|^2  < +\infty,
  \end{equation}
and letting $I\colon \ms S\to [0,+\infty]$ be the functional defined by
   \begin{equation}\label{ihj}
     I (\pi,Q) := H(\pi_0\vert m) + J(\pi,Q), 
   \end{equation}
  the large deviation principle for the Kac's walk is stated as follows.
  
\begin{theorem}
  \label{t:1}
  Let $\nu^N$ as in Assumption \ref{ass:2}. Then 
   for each closed $C \subset  \ms S$, respectively each open $A\subset  \ms S$, 
   \begin{equation*}
     \begin{split}
      &  \varlimsup_{N\to \infty} \frac 1N \log \bb P^N_{\nu^N}
      \big( (\pi^N,Q^N)\in C \big) \le - \inf_C \,I, \\
      &  \varliminf_{N\to \infty} \frac 1N \log \bb P^N_{\nu^N}
     \big( (\pi^N,Q^N)\in A \big) \ge - \inf_{A\cap \hat{\ms S}} \,I.
     \end{split}    
   \end{equation*}
\end{theorem}
The proof of the upper bound does not rely on  item (iii) of Assumption \ref{ass:1}. In particular it holds also when the collision rate conserves the energy. Likewise, item (iii) in Assumption
\ref{ass:2} is used only in the proof of the lower bound.
We also remark that, if we replace item (iii) in Assumption \ref{ass:1} by the condition $\int \de w' B(v, v_*, w') e^{\eta [|v'|^2 + |v'_*|^2]}\leq C(1 + |v-v_*|^2)$, or the collision $r$ has non degenerate density on $(\bb R^d)^4$, the lower bound holds in the whole $\ms S$.
As we show in Proposition \ref{prop:I1}, the projection of $I$ on the empirical measure coincides with the variational expression in \cite{Re, Le}.

\section{Upper bound}
\label{s:3}
The upper bound is achieved by an established pattern in large deviation theory. We first prove the exponential tightness, which allows us to reduce to compacts. By an exponential tilting of the measure,
we prove an upper bounds for open balls and finally we use a mini-max argument to conclude.
\begin{proposition}[Exponential tightness]
  \label{Ptight}
  There exists a sequence of compacts $K_\ell\subset \ms S$ such that
  for any $N$
  $$\bb P^N_{\nu^N}
  \big( (\pi^N,Q^N)\notin  K_\ell \big) \le  e^{-N\ell}.
  $$
\end{proposition}
By standard compactness criteria (Banach-Alaoglu, Prokhorov and Ascoli-Arzel\`a
theorems),
the proof follows from the bounds in the next
three lemmata. 
\begin{lemma}\label{lemma1}
  Let $\zeta\colon \bb R^d \to [0,\, +\infty)$ be the function $\zeta(v) = |v|^2$. Then
    \begin{equation}
      \lim_{\ell\to+\infty}\varlimsup_{N\to+\infty}\frac 1 N \log{ \bb P_{\nu^N}^N \Big ( \sup_{t\in [0,T]} \pi^N_t(\zeta)}
     \geq  \ell \Big)
      =-\infty.
      \end{equation}
\end{lemma}

\begin{lemma}\label{lemma2}
    \begin{equation}
      \lim_{\ell\to+\infty}\varlimsup_{N\to+\infty}\frac 1 N \log{ \bb P_{\nu^N}^N \Big ( Q^N(1) >  \ell \Big)}
      =-\infty.
      \end{equation}
\end{lemma}  


\begin{lemma}\label{lemma3}
  For each $\epsilon >0$ and  $\phi\colon\bb R^d\to \bb R$ continuous and bounded
  \begin{equation}
    \lim_{\delta\downarrow 0}
    \varlimsup_{N\to+\infty}\frac 1 N \log{ \bb P^N_{\nu^N}\Big(\sup_{t, s \,\in [0,T]\;:|t-s|<\delta} |\pi^N_t(\phi)-\pi^N_s(\phi)|> \epsilon  \Big)}
      =-\infty.
  \end{equation}
\end{lemma}

Lemmata \ref{lemma1} and  \ref{lemma3} imply that if the large deviation upper bound rate function is finite then
$\sup_{t\in[0,T]}\pi_t(\zeta)<+\infty$ and  $\pi \in C([0,T], \ms P_0(\bb R^d))$, i.e. $\pi$ meets the conditions in items (i), (ii) in Definition \ref{def:sac}.

To deal with the initial conditions, as detailed in
Assumption~\ref{ass:2}, we need the following elementary statement whose proof is omitted.
\begin{lemma}\label{lemma0}
Pick $\phi\in C(\bb R^d)$ such that $m(e^{\phi})<+\infty$, and let
$m_\phi$ be the probability on $\bb R^d$ defined by
$m_\phi(\de v) = m(\de v) e^{\phi(v)} /m(e^{\phi})$. Then
\begin{equation}\label{mphi}
\nu^N\left( e^{N \pi^N(\phi)} \right) =
\big(m(e^{\phi})\big)^N \, \frac {f_N^\phi(0)}{f_N(0)}
\end{equation}
where $f_N^\phi$ and  $f_N$ are the densities of
$N^{-1/2}\sum_i v_i$ with $v_i$ i.i.d.
with law $m_\phi$ and $m$ respectively.
\end{lemma}

\begin{proof}[Proof of Lemma \ref{lemma1}]
  Given $\gamma>0$ to be chosen later,  let $\Psi(\bs v)=\gamma\sum_{k} |v_k|^2$ and set
  \begin{equation*}
\bb M^\Psi_t := \exp \Big\{ \Psi(\bs v_t) - \Psi(\bs v_0)-\int_0 ^t \de s\,e^{-\Psi} \ms L_N e^\Psi (\bs v_s)\Big\}.
  \end{equation*}  
By e.g. \cite[App.~1, Prop.~7.3]{KL},  $\bb M^\Psi$ is a positive super martingale, in particular 
for any bounded stopping time $\tau$ and any $\bs v_0\in\Sigma_{N,0}$
$\bb E_{\bs v_0}^N \big[\bb M^\Psi_\tau  \big]\leq 1$.
By simple computations, in view of Assumption \ref{ass:1}, item (iv), we can choose $\gamma>0$ such that  there exists a constant  $c$ such that for any $N > 1$
\begin{equation*}
\sup_{\bs v \in \Sigma_{N,0}} \,e^{-\Psi} \ms L_N e^\Psi (\bs v) \leq c N.
\end{equation*}  

Set $\tau_\ell :=\inf \{t>0 : \pi_t^N(\zeta)> \ell  \}\wedge T$, then
\begin{equation*}\begin{split}
&  \bb P_{\nu^N}^N \Big ( \sup_{t\in [0,T]} \pi^N_t(\zeta) \geq  \ell \Big)
  = \bb P_{\nu^N}^N \big (\tau_\ell < T\big)
  =\bb E _{\nu^N}^N \Big(\bb M^\Psi_{\tau_\ell} \big(\bb M^\Psi_{\tau_\ell}\big)^{-1} \id_{\tau_\ell < T} \Big)\\
  &
  \leq \bb E _{\nu^N}^N \Big(\bb M^\Psi_{\tau_\ell} \exp\big\{-\gamma N \ell + \Psi(\bs v_0) + cN \tau_\ell   \big\} \Big)\\
  &\leq \exp\{-N(\gamma \ell -c T ) \} \int d\nu^N \exp \{\gamma \sum_k |v_k|^2   \}.
\end{split}\end{equation*}

To complete the proof we show that Assumption \ref{ass:2} implies that,  possibly by redefining $\gamma >0$, there exists a constant $c$ such that for any $N$
\begin{equation}\label{cond1}
\int \de \nu^N e^{\gamma \sum_k |v_k|^2  }\leq e^{c N}.
\end{equation}  
In order to prove this bound, we apply Lemma \ref{lemma0} with
$\phi(v)=\gamma |v|^2 +\alpha\cdot v$, where $\alpha\in\bb R^d$ is
chosen so that $m_\phi$ is centered.
Observing that with $\nu^N$-probability one
$\pi^N(\gamma |v|^2 +\alpha\cdot v)=\gamma \pi^N(|v|^2)$, we get
\begin{equation*}
  \int \de \nu^N e^{\gamma \sum_k |v_k|^2  } 
  = \big(m(e^{\phi})\big)^N \, \frac {f_N^\phi(0)}{f_N(0)}
\end{equation*}
By Assumption \ref{ass:2}, the densities $f_N^\phi$ and $f_N$ satisfy
the local central limit theorem, see e.g. \cite[Ch.~XV.5, Thm.~2]{Fe}). In particular $\frac 1 N \log f^\phi_N(0)$ and
$\frac 1 N \log f_N(0)$ vanish as $N\to +\infty$.
The proof or \eqref{cond1} is thus achieved.
\end{proof}


\begin{proof}[Proof of Lemma \ref{lemma2}]
 Given $\ell$, $h>0$, set  $B_{\ell,h}:=\{(\pi,Q):\, \sup_{t}\pi_t(\zeta)\leq h, Q(1)>\ell \}$,
 with  $\zeta(v)= |v|^2$.
 In view of the previous lemma, it is enough
 to show that for each $h>0$ 
 \begin{equation*}
\lim_{\ell\to\infty} \varlimsup_{N\to\infty}\frac 1 N \log  \bb P_{\nu^N}^N \Big (\big(\pi^N, Q^N\big)\in B_{\ell, h} \Big)=-\infty.
 \end{equation*}
 Recall that the scattering rate $\lambda$ has been defined in 
  \eqref{sr}.
  Given a bounded measurable function
  $F\colon [0,T] \times \bb R^{4d}\to \bb R$ such that $F(t; v,v_*, v',v_*')=F(t; v_*,v, v',v_*') =F(t; v,v_*, v'_*,v')$,
  let
    \begin{equation}
      \label{def:lambdaF}
      \lambda^F(t;v,v_*)=\int r(v,v_*; \de v',\de v'_*)e^{F(t;v,v_*;v',v'_*)}.
    \end{equation}
    Denoting by $Q^N_{[0,t]}$ the restriction of the measure $Q^N$ on
    $[0,t]$, and setting $\Lambda(v)=\lambda(v,v)$, $\Lambda^F(t,v)=\lambda^F(t;v,v)$,
    $v\in\bb R^d$,
    the
    process
  \begin{equation}\label{mart1}\begin{split}
      \bb N_t^F = & \exp\Big\{N\Big( Q_{[0,t]}^N(F)-\frac 1 2 \int_0^t \!\de s\,
      \Big[
    \pi_s^N\otimes \pi^N_s\big(\lambda^F -\lambda  \big)
    +\frac 1 N 
    \pi^N_s\big(\Lambda^F-\Lambda\big)\Big] \Big)  \Big\}
  \end{split}\end{equation}
  is a $\bb P^N_{\bs v}$ positive super-martingale for each $\bs v\in \Sigma_{N,0}$, see e.g. \cite[App.~1, Prop.~2.6]{KL}.
  Choosing $F=\gamma$, with $\gamma$ a positive constant, for each $\ell >0$ 
  \begin{equation*}\begin{split}
      \bb P_{\nu^N}^N \Big (\big(\pi^N, Q^N\big)\in B_{\ell, h} \Big)  &
      =
      \bb E_{\nu^N}^N \Big(\bb N_T^\gamma \,\big(\bb N_T^\gamma\big)^{-1} \id_{B_{\ell, h}}(\pi^N, \,Q^N)  \Big) \\ &
      \leq \exp\{-\gamma N \ell + (e^\gamma -1)N C(1+2h) T\},
  \end{split}\end{equation*}  
  where in last inequality we have used \eqref{(v)}.
\end{proof}

\begin{proof}[Proof of Lemma \ref{lemma3}]
  In view of the balance equation \eqref{bal} and Lemma \ref{lemma1},
  it is enough to show that for any $h>0$
  there exists a function $c\colon (0,1)\to \bb R_+$ with $c(\delta)\uparrow +\infty$ as $\delta\downarrow 0$ such that, for any $\epsilon >0$
  \begin{equation*}
    \bb P^N_{\nu^N}\Big(\sup_{t\in [0,T-\delta]} Q^N_{[t, t+\delta]}(1)> \epsilon ,\, \sup_{t\in[0,T]} \pi^N_t(\zeta)\leq h
    \Big)\leq e^{-Nc(\delta)}.
  \end{equation*}
  with $\zeta(v)=|v|^2$.
  By a straightforward inclusion of events, the previous bound follows from
  \begin{equation*}
    \frac 1 \delta \sup_{t\in[0,T-\delta]} \bb P^N_{\nu^N}\Big(Q^N_{[t, t+\delta]}(1)> \epsilon, \, \sup_{t\in[0,T]} \pi^N_t(\zeta)\leq h
    \Big)\leq e^{-Nc(\delta)}.\end{equation*}  
 Consider the super-martingale  \eqref{mart1} with $F=\gamma\,\id_{[t, t+\delta]}$, $\gamma >0$. Using the same argument of the previous lemma and \eqref{(v)} we deduce
\begin{equation*}
  \bb P^N_{\nu^N}\Big(Q^N_{[t, t+\delta]}(1)> \epsilon, \, \sup_{t\in[0,T]} \pi^N_t(\zeta)\leq h
  \Big)\leq
  \exp\Big\{-N\big[\gamma \epsilon - \delta\,(e^\gamma -1) C(1+2h)  \big]  \Big\}.
\end{equation*}   
  The proof is concluded by choosing $\gamma= \log(1/\delta)$.
\end{proof}  

\subsection*{Upper bound on compacts}
Given a  bounded continuous function $\phi$ on $\R^d$
we define the  probability measure $m_\phi$
on $\R^d$ by $m_\phi(\de v) = m(\de v) e^{\phi}/m(e^\phi)$ and
we set $\mu_\phi^N=m_\phi^{\otimes N}$.
Moreover,
recall the definition  \eqref{def:lambdaF}  of $\lambda^F$.

\begin{lemma}
  \label{lemma:up-pf}
  For any $(\phi,F)\in C_b(\R^d)\times C_b(\R^+\times {(\R^{d})}^2\times {(\R^{d})}^2)$
  such that  $m_\phi$ is centered and $F(t;v,v_*,v',v'_*)=F(t;v_*,v,v',v'_*)=F(t;v,v_*,v'_*,v')$, 
  and any measurable subset $B\subset \ms S$
  \begin{equation}
    \label{eq:up-pf}
    \varlimsup_{N\to \infty} \frac 1N \log\bb P^N_{\nu^N}
    \Big(
    (\pi^N,Q^N)\in B \Big)
    \le - \inf_{(\pi,Q)\in B} I_{\phi,F}(\pi,Q),
  \end{equation}
  where
  \begin{equation}
    \label{eq:Ipf}
    I_{\phi,F}(\pi,Q):= \pi_0(\phi) - \log \big( m(e^\phi) \big)+
    Q(F)-\frac 1 2 \int_0^T \de t \, \pi_t\otimes \pi_t(\lambda^F-\lambda).
    \end{equation}
  \end{lemma}
  \begin{proof}
    Consider the perturbed initial distribution
    ${\tilde \nu}^N = \mu_\phi^N (\ \cdot\ \vert \sum_{i} v_i = 0)$.
    Recalling the definition of the super-martingale $\bb N^F_t$ in
    \eqref{mart1}, we write
    \begin{equation*}
\bb P_{\nu^N}
    \Big(
    (\pi^N,Q^N)\in B \Big)=\int \de {\tilde\nu}^N\,\frac{\de \nu^N}{ \de {\tilde\nu}^N} \, \bb E^N_{\bf v}\left(\bb N^F_T \big(\bb N^F_T   \big)^{-1} \id_{B}(\pi^N,Q^N)\right)
    \end{equation*}  
      Recalling that  $\Lambda(v):=\lambda(v,v)$, and using Lemma~\ref{lemma0}, we get 
  \begin{equation*}
  \begin{split}
    &\bb P_{\nu^N}
    \Big(
    (\pi^N,Q^N)\in B \Big)
    \\
    &\le  \sup_{(\pi,Q)\in B} e^{-N\pi_0(\phi)} 
    \frac {f_N^\phi(0)}{f_N(0)}\big( m(e^\phi) \big)^N
    e^{
      -N \big\{ Q(F)-\frac 1 2 \int_0^T\! \de t \,[\pi_t\otimes \pi_t(\lambda^F-\lambda)
        +\frac 1N
        \pi_t (\Lambda^F - \Lambda)]
        \big\}
    }
    \\
    &\phantom{\le}\times
   \bb E_{\nu^N}^N \left(
   \frac {  e^{N\pi^N_0(\phi)}}{ \big( m(e^\phi) \big)^N}
   \frac { f_N(0)}{f_N^\phi(0)}\ 
   \bb N_T^F  \id_{B}(\pi^N,Q^N)
   \right) \\
    &\le
     \sup_{(\pi,Q)\in B} e^{-N\pi_0(\phi)} 
    \frac {f_N^\phi(0)}{f_N(0)}\big( m(e^\phi) \big)^N
    e^{
      -N \big\{ Q(F)-\frac 1 2 \int_0^T\! \de t \,\pi_t\otimes \pi_t(\lambda^F-\lambda)
        +\frac 1N TC \big\}
    }\,
   \bb E_{{\tilde \nu}^N}^N ( \bb N_T^F)
  \end{split}
  \end{equation*}
  where in the last inequality we used  \eqref{(v)}.
  The statement is achieved by observing that
  $\bb E_{{\tilde \nu}^N}^N ( \bb N_T^F)\le 1$, and noting that
  by the local central limit theorem both 
  $\frac 1 N \log f^\phi_N(0)$ and
  $\frac 1 N \log f_N(0)$ vanish as $N\to +\infty$.
\end{proof}

Recall that $H(\cdot|m)$ denotes the relative entropy and let $J$ be the functional defined in \eqref{5}.
\begin{proposition}[Variational characterization of the rate functional]
  \label{p:vr} 
  For any pair
  $(\pi,Q)\in\ms S$ satisfying (i) and (ii) in Definition \ref{def:sac}
  \begin{equation}
    \label{vr}
    \begin{split}
      & H(\pi_0|m) =\sup_\phi \Big\{\pi_0(\phi) - \log\big(
      m(e^\phi)\big)
      \Big\},\\
      & J(\pi,Q) = \sup_F \Big\{Q(F)-\frac 1 2 \int_0^T \de t \, \pi_t\otimes
      \pi_t(\lambda^F-\lambda)\Big\}.
    \end{split}
\end{equation}
In the first formula the supremum is carried out over the continuous
and bounded $\phi\colon \bb R^d\to \bb R$ such that the probability
$m_\phi$ (as defined in Lemma~\ref{lemma0}) is centered.
In the second formula the supremum is carried out over all continuous
and bounded $F\colon [0,T]\times (\bb R^d)^2\times  (\bb R^d)^2 \to
\bb R$ such that $F(t;v,v_*,v',v'_*)=F(t;v_*,v,v',v'_*)=F(t;v,v_*,v'_*,v')$.
\end{proposition}

Since the set of $\pi$ satisfying the condition in Definition \ref{def:sac}, items (i), (ii), is a closed subset of $D([0,T]; \ms P_0(\bb R^d))$,
the previous characterization of the rate functional
readily implies the lower semicontinuity of $I$; moreover, if
$\sup_F\{ Q(F)-\frac 1 2 \int_0^T \de t \, \pi_t\otimes \pi_t(\lambda^F-\lambda) \}
< + \infty$ then $Q\ll Q^\pi$, i.e., $(\pi,Q)\in \ms S_\mathrm{ac}$.

\begin{proof}
  The first statement follows from the variational characterization of
  the relative entropy and the observation that since $\pi_0$ is
  centered it is enough to consider $\phi$ satisfying the stated
  constraint. 

  To prove the second statement, recall the definition of $Q^\pi$
  in \eqref{4} and observe that
  \begin{equation*}
   \frac 1 2 \int_0^T \de t \, \pi_t\otimes \pi_t(\lambda^F-\lambda)
    = Q^\pi \left( e^F - 1 \right).
  \end{equation*}
  This implies that if $\sup_F \Big[Q(F)-\frac 1 2 \int_0^T \de t \, \pi_t\otimes
  \pi_t(\lambda^F-\lambda)\Big]$ is finite, then
  $Q$ is absolutely continuous with respect to
  $Q^\pi$.

  The proof is now completed by a direct computation, see  Lemma 4.4 in \cite{BB}.
\end{proof}

\begin{proof}[Proof of Theorem \ref{t:1}: upper bound] 
  By the exponential tightness in Proposition \ref{Ptight}, to prove the upper bound
  it is enough to show the statement
  for compacts. By Lemma \ref{lemma:up-pf} and a
  mini-max argument, see e.g. \cite[App.2, Lemma~3.2]{KL}, the upper bound holds
  with the functional
  $$
  \hat I(\pi, Q) = \sup_{\phi,F} I_{\phi,F} (\pi, Q).
  $$
  Finally, by Proposition \ref{p:vr}, $\hat I = I$.

\end{proof}

\section{Lower bound}
\label{s:4}
In order to obtain the large deviation lower bound, given $(\pi, Q)$ we need to produces a perturbation of the dynamics such that the law of large number for $(\pi^N, Q^N)$ is  $(\pi, Q)$.
While the compactness of $(\pi^N, Q^N)$ follows from the arguments of the previous section, in order to identify the limit point we need uniqueness of the perturbed Boltzmann-Kac equation that we are able to prove only if the perturbed scattering rate is bounded. Therefore, we shall first prove the lower bound for open neighborhoods of ``nice'' $(\pi, Q)$, and then use a density argument, that will be completed with the restriction that $Q$ has bounded second moment. 

\subsection*{Perturbed Kac walks}


We start by the following law of large numbers for a class of
perturbed Kac's walks.
Consider perturbed time-dependent collision rates $\tilde r_t$, with density $\tilde B_t$, i.e.
\begin{equation}
  \label{rtilde}
  \tilde r_t (v,v_*; \de v',\de v_*') = \tilde B_t(v,v_*;w') \de w',
\end{equation}
which we assume to meet  condition (iv) in Assumption \ref{ass:1} uniformly  for  $t\in [0,T]$, and to satisfy the following extra condition.
There exists $C<+\infty$ such that for any $t, v, v_* \in [0,T] \times (\R^d)^2$
\begin{equation}\label{v'}
\tilde\lambda_t(v, v_*)=\int \tilde r_t (v,v_*; \de v',\de v_*') \leq C.
\end{equation}  



Given a probability $\nu$ on $\Sigma_{N,0}$ we denote by $\tilde{\bb
  P}_\nu^N$ the law of the perturbed Kac walk with initial datum $\nu$.

\begin{lemma}\label{l:perk}
  Fix a sequence of initial conditions $\nu^N$ as in Assumption
  \ref{ass:2}.  As $N\to \infty$, the pair $(\pi^N,Q^N)$ converges,
  in $\tilde{\bb P}_{\nu^N}^N$ probability, to
  $(f \de v\,, q\, \de t \de v\de v_* \de w')$, where
  $q_t(v,v_*,w') =\frac 1 2  f_t(v) f_t(v_*) \tilde B_t(v,v_*,w')$ and $f\in C\big([0,T]; L^1(\bb R^d)\big)$ is the
  unique  solution   to the
  perturbed Kac's equation
  \begin{equation}
    \label{perk}
    \begin{array}{l}
      \vspace{3pt}
             {\displaystyle \partial_t f_t(v) =
               \int \!\!\de v_* \de w' \,
      \big[\tilde B_t(v',v'_*;w) f_t(v') f_t(v_*')  -
        \tilde B_t(v,v_*;w')  f_t(v) f_t(v_*) \big]}, \\
      f_0(\cdot) = \frac{\de m}{\de v} \,, 
     \end{array} 
  \end{equation}  
  where $v'= \frac{v+v_*}2 + \frac{w'}{\sqrt{2}}$, $v_*'= \frac{v+v_*}2 -
  \frac{w'}{\sqrt{2}}$, and $w = \frac{v-v_*}{\sqrt{2}}$.
  Here we understand that \eqref{perk} holds by integrating against  continuous, bounded test functions which are continuous differentiable in time.
\end{lemma}  

\begin{proof}
  Observe that the large deviation upper bound proven in the previous
  section holds also for the perturbed Kac's walk. The exponential
  tightness implies that the sequence
  $\{\widetilde{\bb P}_{\nu^N}^N\circ (\pi^N,Q^N)^{-1}\}$ is
  precompact in $\ms P_0 (\ms S)$.  Moreover, by the large deviation
  upper bound, any cluster point $\tilde{\mc P}$ of this sequence
  satisfies
  $\tilde{\mc P}\big( \big\{ (\pi,Q) \colon \tilde I(\pi,Q) =0 \}
  \big)=1$ where the rate function $\tilde I$ is defined as $I$ in
  \eqref{5}, \eqref{ihj},  with the rate $r$ replaced by the perturbed rate
  $\tilde r$.  Since $\tilde I (\pi, Q)< +\infty$ we get
  $(\pi, Q)\in \ms S _{\mathrm{ac}}$, in particular
  $\pi_t(\de v) = f_t(v) \de v$ and
  $Q(\de t;\de v,\de v_*,\de v',\de v_*') = q_t(v,v_*;w') \de t  \de v \de v_* \de w'$ for
  some densities $f$ and $q$. Then $\tilde I (\pi, Q)=0$ implies
  $Q=\tilde Q ^\pi$, where $\tilde Q ^\pi$ is defined as in \eqref{4}
  with $r$ replaced by $\tilde r$. The balance equation \eqref{bal}
  thus amounts to the weak formulation of \eqref{perk}.
  
  It remains to show that $f\in C([0,T]; L^1(\bb R^d))$ and that the solution to \eqref{perk} is unique.
  Choosing test functions independent of time and integrating \eqref{perk} we deduce that for each $t\in[0,T]$ and Lebesgue almost every $v$ it actually holds
  \begin{equation}
    \label{bint}
    \begin{aligned}
      & f_t (v) = f_0(v)
      \\ &\quad +
    \int_0^t\!\de s \int \!\de v_* \de w' \,
      \big[\tilde B_s(v',v'_*;w) f_s(v') f_s(v_*')  -
        \tilde B_s(v,v_*;w')  f_s(v) f_s(v_*) \big].
     \end{aligned} 
    \end{equation}
  Since $\tilde \lambda$ is bounded, it is now straightforward to show that $f\in C([0,T]; L^1(\bb R^d))$.
  Indeed, letting $C$ be the constant in \eqref{v'}, 
  from \eqref{bint} we get $\|f_t-f_s  \|_{L_1}\leq 2C (t-s)$ for $0\leq s\leq t\leq T$.
  Finally, using again \eqref{v'},  uniqueness is achieved
  by applying Gronwall's lemma to \eqref{bint} .
\end{proof}

The collection of ``nice'' $(\pi, Q)$ is specified as follows.
\begin{definition}\label{def:B}
Let  $\tilde{\ms S}$ be the collection of elements
  $(\pi, Q)\in \ms S_{\mathrm{ac}}$ whose densities $(f, q)$ are such
  that
\begin{equation}\label{def:B1}
  {\operatorname{ess}\sup}_{t, v,v_*, w'}\frac {q_t(v,\,v_*; w')}{f_t(v) f_t(v_*)}< +\infty,
\end{equation}
and
\begin{equation}\label{def:B2}
{\operatorname{ess}\sup}_{t, v,v_*}\int \de w' \frac {q_t(v,\,v_*; w')}{f_t(v) f_t(v_*)}e^{\eta |w'|^2}< +\infty,
\end{equation}
for some $\eta>0$.
\end{definition}

Given $(\pi, Q)\in \tilde{\ms S}$, denote by $\tilde r_t$ the time
dependent perturbed rate whose density is defined by
\begin{equation}\label{def:Bt}
\tilde  B_t(v, v_*, w') = \frac {2\,q_t(v,\,v_*; w')}{f_t(v) f_t(v_*)},
\end{equation}
that  meets
the condition (iv) in Assumption \ref{ass:1} uniformly for $t\in [0,T]$ and the extra assumption  \eqref{v'}.

The next statement provides the large deviation lower bound for  neighborhood of elements in $\tilde{\ms S}$. 
\begin{proposition}\label{l:lb}
  Let  $(\pi, Q)\in \tilde{\ms S}$. Assume that $\pi_0$ satisfies items
(i), (ii) in Assumption \ref{ass:2}, and suppose $\pi_0(\de v)=e^\phi
m(\de v)/m(e^\phi)$ for some $\phi$ bounded and continuous.
Moreover,
denote by
$\tilde \nu^N=\pi_0^{\otimes N}(\cdot\vert \sum_i v_i=0)$ the corresponding probability on $\Sigma_{N,0}$. Then
$$
\varlimsup_{N\to\infty}\frac 1 N H\Big(\tilde{\mathbb P}^N_{\tilde \nu^N}\vert \mathbb P^N_{\nu^N} \Big) =
I(\pi, Q).
$$
\end{proposition}  
We premise the following lemma.

\begin{lemma}
  \label{lemmanext}
  Set $\zeta(v) = |v|^2$ and $\bar F (t,v,v_*,w')=\bar F(w')=1+ |w'|^2$, then
  \begin{equation*}
    \begin{aligned}
    &\sup_{N}\,  \tilde {\bb E}_{\tilde\nu^N}
\big(\sup_{t\in [0,T]}\pi^N_t(\zeta)\big)  < +\infty, \\
&    \sup_N \, \tilde {\bb E}_{\tilde\nu^N} \big( Q^N(\bar F)^2 \big) < +\infty.
\end{aligned}
\end{equation*}
  
\end{lemma}
\begin{proof}
  We write
  \begin{equation*}
\tilde {\bb E}_{\tilde\nu^N}
\big(\sup_{t\in [0,T]}\pi^N_t(\zeta)\big) = \int_0 ^\infty \!\!\de \ell\, \tilde{\bb P}_{\tilde\nu^N}\big(\sup_{t\in [0,T]}\pi^N_t(\zeta)>\ell \big).
  \end{equation*}  
 The first bound in the statement is achieved by observing   that Lemma \ref{lemma1}  holds also for the perturbed chain. 

 In order to prove the second bound, let
  \begin{equation*}
  \tilde{M}^{N}_t:= Q^N_{[0,t]}(\bar F)-\frac 1 {N^2}\sum_{\{i,j\}}\int_0^t \!\de s\,\int\! \de w'\, \tilde B_s (v_i,v_j, w') \bar F(w'),
  \end{equation*}
  that it is a $\tilde {\bb P}_{\tilde\nu^N}$ martingale, with predictable quadratic variation
  \begin{equation*}
   \langle \tilde{M}^{N}  \rangle_t = \frac 1 {N^2}\sum_{\{i,j\}}\int_0^t \!\de s\,\int\!\de w'\, \tilde B_s (v_i,v_j, w') \bar F(w').
   \end{equation*} 
    In view of \eqref{def:Bt} and  \eqref{def:B2} in Definition \ref{def:B},  the random variable $\langle \tilde{M}^{N}  \rangle_T $ is uniformly bounded in $N$. This completes the proof. 
  
\end{proof}

\begin{proof}[Proof of Proposition \ref{l:lb}]
  We first prove that
  \begin{equation}\label{ent1}
  \lim_{N\to\infty} \frac 1 N H(\tilde \nu^N \vert \nu^N) = H(\pi_0\vert m).
  \end{equation}
  By Lemma \ref{lemma0},
  \begin{equation*}
    \frac 1 N H(\tilde \nu^N \vert \nu^N)=\tilde\nu^N\big(\pi^N(\phi)  \big)
    -\log m(e^\phi)+ \frac 1 N \log \frac{f_N(0)}{f_N^\phi(0)},
  \end{equation*}  
where, by the local central limit theorem, the last term on the right hand side vanishes as $N\to+\infty$.
As a corollary of Lemma \ref{l:perk} we deduce that $\pi^N$ converges in $\tilde \nu^N$-probability to $\pi_0$. Hence, in view of assumptions on $\phi$, we deduce
$$
\lim_{N\to\infty} \frac 1 N \,H(\tilde \nu^N \vert \nu^N)=\pi_0(\phi)
-\log m(e^\phi)= H(\pi_0\vert m).$$
We now show that
\begin{equation}\label{ent2}
  \varlimsup_{N\to\infty}\frac 1 N\, H\big(\tilde{\bb P} ^N_{\tilde\nu^N}\vert  \bb P ^N_{\tilde\nu^N}\big)= J(\pi,Q).
\end{equation}
In view of the assumptions on $\tilde r_t$ the super-martingale defined in \eqref{mart1} with $F_t=\log(\de \tilde r_t/ \de r)$ is actually a martingale and its value at time $T$ is the Radon-Nykodim derivative of $\tilde{\bb P} ^N_{\tilde\nu^N}$ with respect to $\bb P ^N_{\tilde\nu^N}$. Since $\lambda_t^F=\tilde \lambda_t$ and $\Lambda^F_t(v)=\tilde\lambda_t(v,v)=:\tilde\Lambda_t$,
\begin{equation*}
  \frac 1 N \, H\big(\tilde{\bb P} ^N_{\tilde\nu^N}\vert  \bb P ^N_{\tilde\nu^N}\big)
  =\tilde{\bb E}_{\tilde\nu^N}\Big (Q^N (F)-\frac 1 2\int_0^T\!\! \de t\,\big[ \pi_t^N\otimes \pi_t^N(\tilde\lambda_t -\lambda)
  +\frac 1 N  \pi^N_t(\tilde\Lambda_t-\Lambda)\big]\Big).
\end{equation*}

By definition of $\tilde {\ms S}$
\begin{equation*}
  F_t=\log \frac{2\, q_t(v,v_*, w')}{f_t(v)f_t(v_*) B(v,v_*,w')}\leq C (1 + |w'|^2).
\end{equation*}  
where we used Assumption \ref{ass:1}, item (iii).

Now observe that, by Lemma \ref{l:perk}, $(\pi^N,Q^N)$  converges to $(\pi,Q)$ in $\tilde{\bb P}^N_{\tilde \nu^N}$ probability.
Moreover, Lemma \ref{lemmanext} provides sufficient conditions for the uniform integrability of $Q^N(F)$ and of   $\int_0^T\!\! \de t\, \pi_t^N\otimes \pi_t^N(\lambda)$.
Finally, by the boundedness of $\tilde\lambda_t$ and the absolutely continuity of $\pi_t$ we obtain
$$
\varlimsup_{N\to\infty}\frac 1 N \, H\big(\tilde{\bb P} ^N_{\tilde\nu^N}\vert  \bb P ^N_{\tilde\nu^N}\big)
=
Q(F)-\frac 1 2 \int_0^T\!\! \de t\, \pi_t\otimes \pi_t(\lambda^F -\lambda)=J(\pi,Q).
$$
Recalling \eqref{ihj}, the statement follows from \eqref{ent1} and \eqref{ent2}.

\end{proof}

By general results, see e.g.  \cite{Mar}, the previous proposition implies
the following lower bound statement.

\begin{corollary}\label{cor:1}
  Let $\nu^N$ be as in Assumption \ref{ass:2} and set
  \begin{equation*}
    \tilde I (\pi,Q)=\begin{cases}
    I(\pi, Q) & \mbox{if } (\pi, Q)\in \tilde{\ms S}\\
    +\infty & \mbox{otherwise}.
    \end{cases}
  \end{equation*}
  Then the sequence $\big\{\bb P^N_{\nu^N}\circ (\pi^N,Q^N)^{-1}\big\}$ satisfies a large deviations lower bound  with
    rate function $\tilde I $.
\end{corollary}  

\subsection*{Approximating paths with bounded  rate function}  %
  %
   Recall that  the set $\hat{\ms S}$ has been defined in \eqref{QV}.

\begin{theorem}\label{the:approx}
  %
  %
   
  For each  $(\pi, Q)\in \hat{\ms S}$ such that  $I(\pi,Q) < +\infty$
  there exists a sequence $\{(\pi_n,Q_n)\}\subset \tilde{\ms S}\cap \hat{\ms S}$ satisfying
  $(\pi_n,Q_n)\to (\pi, Q)$ and $I(\pi_n,Q_n)\to I(\pi, Q)$.
\end{theorem}


The lower bound in Theorem \ref{t:1} follows directly from Corollary \ref{cor:1} and the above theorem.
In order to prove it, we premise the following lemma.
\begin{lemma}\label{l:qlog}
  Let $(\pi, Q)\in \hat{\ms S}$ be such that $I(\pi, Q)<+\infty$. Denote by $(f,q)$ the densities of $(\pi, Q)$.
  Moreover, set $\sigma:=B(v,v_*,w')/g(w')$ and
  $\Phi:= \log \big( \frac {2\, q(v,v_*,w')} {f(v) f(v_*) g(w')}\big)$,
where $g$ is the standard Gaussian density on $\bb R^d$. Then
  \begin{itemize}
  \item[(i)]$\displaystyle \int Q(\de t; \de v,\de v_*,\de v',\de v'_*)\big(|v|^2+|v_*|^2+|v'|^2+|v_*'|^2\big)  < +\infty$.
  \item[(ii)]  $\displaystyle I(\pi, Q)= H(\pi_0\vert m) +  Q(\Phi) - Q(\log \sigma)  -Q(1) + Q^\pi(1)$.
   \end{itemize} 
\end{lemma}

\begin{proof}
  We start by proving
  \begin{equation}\label{dario}
\displaystyle \int Q(\de t; \de v,\de v_*,\de v',\de v'_*)|v'-v_*'|^2  < +\infty.
  \end{equation}
  By Assumption \ref{ass:1}, item (iv), there exists $\eta>0$ such that 
  $$
  \int Q^\pi (\de t; \de v,\de v_*,\de v',\de v'_*)\exp\{\eta |v' - v'_*|^2  \} < +\infty.
  $$
  Recall that $I(\pi, Q)= H(\pi_0\vert m) + J(\pi, Q)$, where $J$ is
defined in \eqref{5}. By the variational representation of $J$ in
Proposition~\ref{p:vr}, 
for any $F$ bounded and continuous
  $$
Q(F)\leq \frac 1 2 \int_0^T \de t\, \pi_t\otimes\pi_t (\lambda^F -\lambda) + J(\pi, Q).
$$
By a truncation argument, we can choose $F=\eta |v'-v'_*|^2$, and thus deduce \eqref{dario}.

Let $\zeta_n(v)=|v|^2\wedge n$. 
By using the continuity equation 
we get
$$
\pi_T(\zeta_n)-\pi_0(\zeta_n)=\int \de Q\big[\zeta_n(v') + \zeta_n(v'_*)-\zeta_n(v)-\zeta_n(v_*)].
$$
By item (ii) in Definition \ref{def:sac}, \eqref{QV}, and \eqref{dario}, taking the limit $n\to \infty$ we deduce
$$
\displaystyle \int Q(\de t; \de v,\de v_*,\de v',\de v'_*)\big(|v|^2 +|v_*|^2 \big) < +\infty.
$$
which, together with the conservation of the momentum, implies the statement (i).
 
  In view of Assumption \ref{ass:1}, items (iii) and (v), the statement (ii) follows from (i).
  
\end{proof}

\begin{proof}[Proof of Theorem \ref{the:approx}]{}$~$
Observe that by the lower semicontinuity of $I$, for any sequence $(\pi_n, Q_n)\to (\pi, Q)$ we have $\varliminf_n I(\pi_n, Q_n)\ge I(\pi, Q)$. The converse inequality is 
achieved by combining steps 1 and 2 below and  a standard diagonal argument.
  
\noindent
{\bf Step 1 - Convolution.}
Since $I(\pi, Q)<+\infty$, there exist $(f,q)$ such that
$d\pi_t=f_t(v)\de v$  and
$dQ= q_t(v,v_*, w')\de t \de v \de v_* \de w'$
where 
$w=(v-v_*)/\sqrt 2$. 

Given $0<\delta<1$, let $g_\delta$ be the Gaussian kernel on $\bb R^d$ with variance $\delta$ and 
define
     \begin{equation*}\begin{split}
         & f^\delta_t(v)=
         \int \de u\, g_\delta (v-u) f_t(u)\\
         & q^\delta_t(v, v_*, w')=
         \int\! \de u \de u_* \de z'\, g_\delta (v-u)
         g_\delta(v_*-u_*) g_\delta(w'-z')q_t(u, u_*, z').
        \end{split} 
     \end{equation*}
     We now show that the pair $(f^\delta, q^\delta)$ satisfies the balance equation. Given a test function $\phi$ and denoting by $*$ the convolution 
     \begin{equation*}\begin{split}
         \int_0^T \de t\int \de v\, \de v_* \de w' \, q^\delta_t(v, v_*,w') \phi_t(v)
       = \int_0^T \de t  \int\! \de v\,\de v_* \de w' \,q_t(v, v_*, w') g_\delta *\phi_t(v).
      \end{split}\end{equation*} 
     One can repeat the same argument with $\phi(v)$ replaced by $\phi(v_*)$.
     Moreover, since $g_\delta(v)g_\delta (v_*)g_\delta(w')=g_\delta(v')g_\delta (v'_*)g_\delta(w)$, where $w=(v-v_*)/\sqrt 2 $, it holds also when $\phi(v)$ is replaced by
     $\phi(v')$, $\phi(v'_*)$. Using the balance equation for the pair $(f,q)$ with the test function
     $g_\delta  *\phi$ we deduce that $(f^\delta, q^\delta)$ satisfies the balance equation.

Now we show that
\begin{equation}  \label{eq:Idelta}
\limsup_{\delta \to 0} I(f^{\delta}, q^{\delta}) \leq I(\pi, Q).
\end{equation}
To this end, we use the decomposition provided by item (ii) of Lemma \eqref{l:qlog}.
We start by observing that, in view of  \eqref{(v)} and
$\sup_{t}\pi_t(\zeta)< +\infty$, $\zeta(v)= |v|^2$, by dominated convergence 
$$
\lim_{\delta\to 0} \frac 12 \int \de t \de v  \de v_* \,f_t^\delta(v)
f_t^\delta(v_*) \lambda (v, v_*) = Q^\pi(1).
$$
Analogously, since $|\log\sigma|\leq C(|v|^2 + |v_*|^2 + |w'|^2)$, in view of Lemma \ref{l:qlog}, item (i)
$$
\lim_{\delta \to 0} \int \de t \de v  \de v_*  \de w' \,q_t^\delta(v, v_*, w')\big(1 + \log \sigma (v, v_*, w') \big)\leq  Q(1)+ Q(\log \sigma).
$$
By the convexity of the map $[ 0, +\infty)^2 \ni (a,b)\mapsto a\log(a/b)$ and Jensen's inequality
  \begin{equation*}
    \int \de t \de v  \de v_*  \de w' \,q_t^\delta (v, v^*, w')\log \frac {2\, q_t^\delta (v, v^*, w')}{f^\delta_t(v)f^\delta_t(v^*)g_{1+\delta}(w')}
    \leq Q(\Phi).
  \end{equation*}  
  By Lemma \ref{l:qlog}, item (i),
  $$
\lim_{\delta\to 0} \int \de t \de v  \de v_*  \de w' \,q_t^\delta (v, v^*, w')\log \frac {g_{1+\delta}(w')}{g_1(w')}=0.
  $$
 Gathering the above statements we deduce
$$
\limsup_{\delta\to 0} J(f^\delta, q^\delta) \leq J(\pi,Q).
$$
To conclude the proof of   \eqref{eq:Idelta}
it remains to show that
\begin{equation}\label{Hd}
\limsup_{\delta\to 0}H(f^\delta_0| m) \le H(f_0|m). 
\end{equation}
By item (iii) in  Assumption \ref{ass:2}, since $f_0$ has bounded second moment, we have
$$
H(f^\delta_0| m)=\int f^\delta_0\log f^\delta_0 + \int f^\delta_0 \log \frac {1} m.
$$
The bound  \eqref{Hd} follows by using item (iii) in Assumption \ref{ass:2} and Jensen's inequality.

\medskip   
\noindent
{\bf Step 2 - Truncation of the flux.}
Given a pair $(f,q)$ that satisfies the balance equation, we denote
by $q_t^{(i)}$, $i=1,\ldots ,4$ the marginal of $q_t$
respectively on $v,v_*,v',v_*'$. Then
$q_t^{(1)}= q_t^{(2)}$,
$q_t^{(3)}= q_t^{(4)}$, and the balance equation reads 
$$\partial_t f_t = 2\big( q_t^{(3)} - q_t^{(1)}\big).$$
In the sequel we assume  $(f,q)$ such that $I(f,q)<+\infty$, $f$ strictly positive on compacts uniformly in time,
and  $q_t^{(3)}\in L^1([0,T]; L^2(\bb R^d))$.
Observe that pair $(f^\delta, q^\delta)$
constructed in Step 1 meets the above conditions. Indeed, by item (ii) in Definition \ref{def:sac}, for each compact subset $K\in\bb R^d$ and $\delta\in (0,1)$,
there exists $c_{K,\delta} >0$ such that, for any $t\in[0,T]$ we have
$\inf_{v\in K} g_\delta * f_t \geq c_{K, \delta}$. Moreover,
by Young inequality, $\int_0^T \de t\,\| q_t^{\delta, (3)}\|_{L^2} \leq c_\delta Q(1)$.

Given $\ell>0$, set $\Omega_\ell$ the subset of $\ms V$
given by
$$ \Omega_\ell  = \{ (V,w,w'):\, |V|^2 +|w|^2+|w'|^2\le \ell^2\},$$ 
and define $(\tilde f^\ell, \tilde q^\ell)$ by
\begin{equation}
  \label{tildeell}
  \begin{aligned}
    &\tilde q^\ell = (q\wedge \ell) \id_{\Omega_\ell}\\
    &\tilde f^\ell_t = f_0 +
    2\left(
      \int_0^t \de s
      \left(\tilde q^{\ell,(3)}_s-
        \tilde q^{\ell,(1)}_s
      \right)
      + \int_0^T  \de s
      \left(q^{(3)}_s-
        \tilde q^{\ell,(3)}_s
      \right)
    \right)
    \id_{|v|\le \ell}.
  \end{aligned}
\end{equation}
Set  $(f^\ell, q^\ell) = c_\ell (\tilde f^\ell, \tilde q^\ell)$
where
$$c_\ell^{-1} = 1 +  2\int_0^T  \de s \int_{|v|\le \ell} \de v
\left(q^{(3)}_s-
  \tilde q^{\ell,(3)}_s
\right) 
$$
Observe that
\begin{equation}
  \label{variqtilde}
  \begin{aligned}
    &\int_0^t \de s
    \left(\tilde q^{\ell,(3)}_s-
      \tilde q^{\ell,(1)}_s
    \right)
  + \int_0^T  \de s
  \left(q^{(3)}_s-
    \tilde q^{\ell,(3)}_s
  \right)\\
  &=
  \int_0^t \de s
  \left( q^{(3)}_s-
    \tilde q^{\ell,(1)}_s\right)
    +
    \int_t^T  \de s
    \left(q^{(3)}_s-
      \tilde q^{\ell,(3)}_s
    \right).
    \end{aligned}
\end{equation}
The previous identity implies that  if $|v|\le \ell$, then  $f^\ell_t \ge c_\ell f_t $, while
 if $|v|> \ell$, then  $f^\ell_t = c_\ell f_0$.
In particular, for any $t$, $f_t^\ell>0$, and $\int f_t^\ell = 1$.
Observe that by construction the pair $(f^\ell, q^\ell)$ satisfies
the balance equation and it is an element of $\tilde {\mathcal S}$
(see definition \ref{def:B}), since $q^\ell$ is bounded and compactly supported and $f$ is strictly positive on compacts, uniformly in time.
Moreover, $(f^\ell, q^\ell)$ converges to $(f,q)$. 

Now we prove that
$$\varlimsup_{\ell \to +\infty}
I(f^\ell, q^\ell) \le I(f,q).$$
We start by proving that
\begin{equation}\label{dario3}
\varlimsup_{\ell \to +\infty}
H(f_0^\ell|m) \le H(f_0|m).
\end{equation}
By definition
$$f_0^\ell = c_\ell f_0 + (1-c_\ell) \bar h^\ell$$
where
$\bar h^\ell = \frac { h^\ell} {\int h^\ell}$ where
$$h^\ell = 
2\left( \int_0^T  \de s
  \left(q^{(3)}_s-
    \tilde q^{\ell,(3)}_s
  \right) \right)     \id_{|v|\le \ell}.
$$
Since $c_\ell\to 1$, by the convexity of $H(\cdot | m)$ it is enough to show
$$(1-c_\ell) H(\bar h^\ell|m)\to 0.$$
Observe that
$$
(1-c_\ell) H(\bar h^\ell|m)= c_\ell \int h^\ell \log h^\ell +(1-c_\ell)\log \frac{c_\ell}{1-c_\ell} + c_\ell \int h^\ell \log \frac 1 m.
$$
Since, by assumption on $q^{\ell, (3)}$, $h^\ell \in L^2$ and it converges to zero point-wise, using item (iii) of Assumption \ref{ass:2}, we deduce \eqref{dario3} by dominated convergence.

It remains to show that 
\begin{equation}\label{dario4}
\varlimsup_{\ell \to +\infty}
J(f^\ell, q^\ell) \le J(f,q).
\end{equation}
Recalling the scattering rate $\lambda$ defined in \eqref{sr}, we rewrite
\begin{equation*}
  J(f^\ell, q^\ell)=\int_0^T\!\!\de t\int_{\Omega_\ell}
  \de v \de v_* \de w'\Big\{ q^\ell\log\frac{2\, q^\ell}{f^\ell f_*^\ell B} -
  q^\ell + \frac 12 f^\ell f^\ell_* B\Big\}.
\end{equation*}
Since $ \int q|\log (2\, q/ff_*B)| <+\infty $, using the bound $f^\ell \geq c_\ell f$ for $|v|\leq \ell$, by dominated convergence
$$
\lim_{\ell\to\infty} \int_0^T\!\!\de t\int_{\Omega_\ell} \de v\de v_*\de w'\,
q^\ell\log\frac{2\, q^\ell}{f^\ell f_*^\ell B} = \int_0^T\!\!\de t\int
\de v\de v_*\de w' \,q\log\frac{2\, q}{f f_* B}.
$$
Recalling \eqref{(v)}, item (i) in Lemma \ref{l:qlog} implies the uniform integrability of $\lambda$ with respect to $f^\ell f^\ell_*$, hence
\begin{equation}\label{ui}
\lim_{\ell\to \infty} \frac 12 \int_0^T\!\!\de t\int_{\Omega_\ell}
\de v\de v_*\de w' \,f^\ell f^\ell_* B =
\frac 12 \int_0^T\!\!\de t\int  \de v\de v_*\, ff_* \lambda, 
\end{equation}
which concludes the proof of \eqref{dario4}. 
\end{proof}

\subsection*{Example of $(\pi, Q)\in \ms S\setminus \hat{\ms S}$ with finite rate function}
In Theorem \ref{t:1} the upper and lower bounds match only for $(\pi, Q)\in \hat{\ms S}$. Here we provide an example of $(\pi,Q)\in \ms S\setminus \hat{\ms S}$ with finite cost. 

Fix $T>1$. Recalling that $g$ is the standard Gaussian density on $\bb R^d$, consider the pair $(f, q)$ given by
\begin{equation*}
  q_t(v, v_*, w')=\frac 1 2 
  \begin{cases}\vspace{0.2cm}
   \displaystyle A g(v) g(v_*)\frac 1 {1+|w'|^{d+3}}, & t\in[0,1)\\
    \displaystyle  \frac 1 {1+|v+v_*|^{d+1}}g((v-v_*)/\sqrt2)g(w'), & t\in [1, T],
  \end{cases}  
\end{equation*}  
where $A^{-1}=\int \de w'\frac 1 {1+|w'|^{d+3}} $, and
\begin{equation*}
  f_t(v)=
  \begin{cases}
    (1-t)g(v) + t h(v), & t\in[0,1),\\
    h(v) & t\in [1,T],
  \end{cases}  
\end{equation*}  
where
\begin{equation*}\begin{split}
    h(v)= &A\int \de v_* \de w' \,g\big( \tfrac {v+v_*}{2}+ \tfrac {w'} {\sqrt 2}\big) g\big( \tfrac {v+v_*}{2}- \tfrac {w'} {\sqrt 2} \big)\frac 1 {1+ \big(|v-v_*|/\sqrt2\big)^{d+3}}\\
         =&    2^d A \int \de u\, g\big ( \sqrt 2 u\big)\frac 1 {1 + (\sqrt 2|v-u|)^{d+3}}.
\end{split}\end{equation*}
For $t\in [1,T]$, $q_t$ is invariant with respect to $(v, v_*, v', v_*')\mapsto (v', v'_*, v, v_*)$. Hence,
by construction, the pair $(\pi, Q)$ whose densities are $(f,q)$ satisfies the balance equation \eqref{bal}. Moreover, $(\pi, Q)\in \ms S_{\textrm{ac}}$ and $(\pi, Q)\notin \hat{\ms S}$.
We next show that $I(\pi, Q)<+\infty$. By item (iii) in Assumption \ref{ass:2}, $H(g|m)< +\infty$ and, by construction, $Q(1)$ and $Q^{\pi}(1)$ are both finite.

We observe that, since $f_t(v)\geq (1-t) g(v)$, $t\in [0,1]$, by item (iii) in Assumption \ref{ass:1}
\begin{equation*}\begin{split}
  & \int_0^1 \de t  \int \de v \de v_* \de w' \,q_t(v,v_*,w')\log \frac {2\, q_t(v, v_*,w')} {f_t(v) f_t(v_*) B(v,v_*,w')}\\
  &\leq \frac A 2 \int_0^1 \de t  \int \de v \de v_* \de w' \,g(v)g(v_*)\frac 1 {1+|w'|^{d+3}}\log \frac {e^{c_0|w'|^2}} {(1+|w'|^{d+3} )(1-t)^2 g(v)g(v_*) c_0},
\end{split} \end{equation*}  
which is finite. Moreover, for $t\in [1,T]$
\begin{equation*}\begin{split}
    & 
    \,q_t(v,v_*,w')\log \frac {2\, q_t(v, v_*,w')} {f_t(v) f_t(v_*) B(v,v_*,w')}\\
    &\leq
    \frac 1 2 \frac 1 {1+|v+v_*|^{d+1}}g((v-v_*)/\sqrt2)g(w')\\
    & \quad \times\log  \frac {g((v-v_*)/\sqrt2)g(w')e^{c_0|w'|^2} } {(1+|v+v_*|^{d+1})g((v-v_*)/\sqrt2) h(v)h(v_*)c_0}.
\end{split}\end{equation*}  
Since there exists $C$ such that
$
h(v)\geq  C \frac 1 {1+ |v|^{d+3}},
$
we deduce
$$
\int_1^T dt  \int \de v \de v_* \de w' \,q_t(v,v_*,w')\log \frac {2\, q_t(v, v_*,w')} {f_t(v) f_t(v_*) B(v,v_*,w')} <+\infty.
$$
The previous bounds imply $I(\pi, Q)<+\infty$.

\section{Projection on the empirical measure}
In this section we analyze the large deviation asymptotics of the empirical measure only. By contraction principle, the corresponding rate function is obtained by projecting the joint rate function $I$.
Regarding the upper bound, we prove that this projection corresponds to the rate function in \cite{Le, Re, Bou}.
For the lower bound, we identify the projection of $I$ only for suitable $\pi$.

Let $I_1\colon D([0,T]; \ms P_0(\bb R^d))\to [0, +\infty]$ be defined by
$I_1(\pi)= H(\pi_0|m)+ J_1(\pi)$, when $\pi$ meets conditions (i) and (ii) in Definition \ref{def:sac}, and $I_1(\pi)=+\infty$ otherwise. Here 
\begin{equation}\label{def:J1}
J_1(\pi)=\sup_{\phi}\Big\{\pi_T(\phi_T)-\pi_0(\phi_0)-\int_0^T\de t\, \pi_t(\partial_t \phi) -Q^\pi\big(e^{\bar\nabla\phi}-1 \big) \Big\},
\end{equation}  
where $\bar\nabla \phi =\phi(v')+\phi(v'_*)-\phi(v)-\phi(v_*)$ and the supremum is carried over the functions $\phi\colon [0,T]\times \bb R^d\to \bb R$ continuous, bounded and continuously differentiable in $t$ with bounded derivative. 

\begin{proposition}\label{prop:I1}
  For any $\pi\in D([0,T]; \ms P_0(\bb R^d))$
  \begin{equation}\label{ub:J1}
 \inf_{Q}I(\pi, Q)\geq I_1(\pi).
  \end{equation}
  Moreover, if $\pi$ is such that the supremum in \eqref{def:J1} is achieved,  then
  \begin{equation}\label{J1}
 \inf_{Q}I(\pi, Q) =  I_1(\pi).
  \end{equation}
\end{proposition}  

\begin{proof}
  Recalling Proposition \ref{p:vr}, the proof of \eqref{ub:J1} is achieved by choosing $F=\bar\nabla \phi$ in \eqref{vr}.
  To prove the second statement, we first note that if the supremum is achieved at some $\phi$, then for any $\psi$ continuous, bounded and continuously differentiable in $t$ with bounded derivative
  $$
\pi_T(\psi_T)-\pi_0(\psi_0)-\int_0^T\de t \,\pi_t(\partial_t \psi)=Q_\phi(\bar\nabla\psi), \quad \de Q_\phi:= \de Q^\pi e^{\bar\nabla\phi}. 
  $$
Recalling \eqref{5}, by choosing $Q=Q_\phi$, a direct computation implies $\inf_Q J(\pi, Q)\leq J(\pi, Q_\phi)=J_1(\pi)$.
\end{proof}

\section{Gradient flow formulation of the
  Boltzmann-Kac equation} 
\label{s:5}
Assuming the detailed balance condition, here we derive the gradient flow formulation of the Boltzmann-Kac equation \eqref{eq:Bgen} associated to the large deviation rate function \eqref{Jr}. We remark that such formulation is logically independent from the validity of the large deviation principle.

Let  $M$ be  the standard Maxwellian on $\bb R^d$. In this
section we assume that the collision rate $r$ satisfies the following
detailed balance condition
\begin{equation}\label{dbc}
  M(\de v) M(\de v_*) r(v,v_*;\de v',\de v_*')=M(\de v') M(\de v_*')r(v',v_*';\de v,\de v_*).
\end{equation}
This implies that the Kac walk on $\Sigma_{N}$ is reversible with
respect to the product measure $\prod_{k=1}^NM(\de v_k)$.  We still
consider the Kac walk restricted to $\Sigma_{N,0}$, then the
corresponding reversible measure is the product measure
$\prod_{k=1}^NM(\de v_k)$ conditioned to $N^{-1}\sum_k v_k =0$, that is
a Gaussian measure on $\Sigma_{N,0}$.

In this section we express the empirical measure and flow in terms of their densities with respect to  Maxwellians. 

Let $\mathcal H \colon \ms P_0(\bb R^d)\to [0, +\infty]$ be the
relative entropy with respect to $M$, i.e.\
$\mathcal H(\pi):= H(\pi\vert M)$. For $\pi\in \ms P_0(\bb R^d)$
with bounded second moment, define the non linear Dirichlet form
$\mc D\colon \ms P_0(\bb R^d)\to [0, +\infty]$ as the lower
semicontinuous map defined by
  \begin{equation}
    \label{vardir}
    \mathcal D (\pi) :=\sup_{\phi}\Big\{\int \pi(\de v)\pi (\de v_*)r(v,v_*, \de v', \de v'_*) \big(1- e^{\phi(v')+ \phi(v'_*)- \phi(v)-\phi(v_*)}  \big)  \Big\}
  \end{equation}
  where the supremum is carried out over the continuous and bounded functions
  $\phi\colon \bb R^d\to \bb R$.
  Note that $\mathcal D(\pi)$ is well defined in view of \eqref{(v)}.
    To illustrate this definition, consider the Markov generator $\mc L$
  acting on functions $\xi \colon \bb R^d \times \bb R^d\to \bb R$ as 
\begin{equation*}
  \mc L \xi (v,v_*) = \int r(v,v_*;\de v',\de v_*')
  \big[ \xi(v',v_*') - \xi(v,v_*) \big].
\end{equation*}
By the detailed balance condition, $\mc L$ is reversible with respect
to the product measure $M(\de v)\, M(\de v_*)$.
The variational representation \eqref{vardir} thus corresponds to the
Donsker-Varadhan functional $\mc E$, that it is defined on the probabilities
on $\bb R^d \times \bb R^d$ by
\begin{equation*}
  \mc E (\Pi) = \sup_{\xi} \Big\{ - \int\! \Pi(\de v,\de v_*) \, e^{-\xi}
  \mc L e^{\xi} \Big\},
\end{equation*}
where the supremum is carried out over the continuous and bounded
functions $\xi\colon \bb R^d\times \bb R^d\to \bb R$.  Then
$\mc D (\pi) = \mc E(\pi\times \pi)$, observe indeed, as proven in
Lemma~\ref{t:rdf} below, that for product
measures $\Pi$ we can restrict the class test functions
$\xi$ to functions of the form
$\xi(v,v_*) = \phi(v)+\phi(v_*)$.

On the set of functions $G\colon \bb R^d \times \bb R^d \times \bb R^d
\times \bb R^d \to \bb R$, let $\Upsilon$ be the involution defined by 
$(\Upsilon G) \,(v,v_*,v',v_*') := G (v',v_*',v,v_*)$.  
Recalling the definition of $Q^\pi$ in \eqref{4}, we define the
\emph{kinematic term} as the lower semicontinuous  functional
$\mc R $ 
 on the  pairs $(\pi, Q)$ satisfying conditions (i) and (ii) in Definition \ref{def:sac}
\begin{equation}
  \label{varkin}
  \mc R (\pi,Q) := \sup_{\alpha,F} \Big\{ 2\, Q(F)
  -Q^\pi \Big(  \big[e^F-1\big] \alpha^{-1} +
  \big[ e^{\Upsilon F} -1\big] \Upsilon\alpha \Big)\Big\}.
\end{equation}
where the supremum is carried out over the bounded and continuous
$F,\alpha \colon [0,T] \times \big(\bb R^d \big)^2 \times \big(\bb R^d \big)^2 \to \bb R$
satisfying $F(t,v,v_*,v', v'_*)=F(t,v_*,v,v', v'_*)=F(t,v,v_*,v'_*, v')$, $\alpha(t,v,v_*,v', v'_*)=\alpha(t,v_*,v,v', v'_*)=\alpha(t,v,v_*,v'_*, v')$,
and  $\inf \alpha >0$.

The main result of this section provides, when the detailed balance 
condition \eqref{dbc}  holds, a gradient flow formulation of the Boltzmann-Kac
equation. Recall that the functionals $J$ and $I$ have been
introduced in \eqref{5} and \eqref{ihj} and that $\hat{\ms S}$ is the set of paths $(\pi, Q)\in\ms S_{\mathrm {ac}}$ that satisfy \eqref{QV}.

\begin{theorem}
  \label{gf}
  Assume that $\mc H(\pi_0) < +\infty$. For each $(\pi,Q)\in \hat{\ms S}$
  \begin{equation}
    \label{J=}
    J(\pi,Q) = \frac 12 \big[ \mc H(\pi_T) - \mc H(\pi_0)\big]
    + \frac 12 \int_0^T\!\de t\, \mc D (\pi_t) + \frac 12 \mc R (\pi,Q).
  \end{equation}
  In particular, when the scattering rate $\lambda$ is bounded,
  $I(\pi,Q)=0$ if and only if $\pi_0=m$ and
  \begin{equation}\label{epi}
    \mc H(\pi_T)  + \int_0^T\!\de t\, \mc D (\pi_t) + \mc R (\pi,Q) \le
    \mc H(m). 
  \end{equation}
\end{theorem}

We start by the following characterization of the Dirichlet form $\mc D$
and the kinematic term $\mc R$ in which we recall that $\mc V$ is the hyperplane of $(\bb
R^d)^2\times (\bb R^d)^2$ defined by $v+v_*= v'+v_*'$ and   
$r(v,v_*;\de v_*,\de v_*')= \sigma(v,v_*,w') M(\de w')$ where
$w'=(v'-v_*')/\sqrt{2}$.

\begin{lemma}
  \label{t:rdf} 
  Let  $\pi\in \ms P_0(\bb R^d)$ be  such that $\pi(\de v) = h(v) M(\de v)$,
  $\pi(\zeta) <+\infty$, $\zeta(v) = |v|^2$, and
  $\mc D(\pi) <+\infty$.  Then
  \begin{equation}\label{Dir1}
  \sqrt{ h(v)h(v_*)h(v')h(v_*')} \, \sigma(v,v_*,w') \in L^1\big(\mc V,
  M(\de v) M(\de v_*) M(\de w')\big)  
  \end{equation}
  and
  \begin{equation}\label{Dir2}
    \begin{split}
    \mc D(\pi) &=   \int\!M(\de v) M(\de v_*)M(\de w') \, h(v)h(v_*) \sigma(v,v_*,w')
    \\ &\quad - \int\!M(\de v) M(\de v_*)M(\de w') \, \sqrt{ h(v)h(v_*)h(v')h(v_*')}
    \, \sigma(v,v_*,w'). 
  \end{split}  \end{equation}
  
  Moreover, set $d\pi_t= h_t \de M$, $\de Q=\de t\,p_t(V,w,w') M(\de V) M(\de w) M(\de w')$, and
$ r(v,v_*;\de v',\de v_*') = \bar\sigma(V;w,w') M(\de w')$ where
$V=(v+v_*)/\sqrt{2}$, $w=(v-v_*)/\sqrt{2}$, and
$w=(v'-v_*')/\sqrt{2}$. Then, if $\mc R(\pi, Q) <+\infty$,
\begin{equation}\label{Re}
  \begin{split}
  \mc R (\pi,Q) =& 2 \int \! \de t M(\de V) M(\de w) M(\de w')\,
  \\
  &  \times \Big[ p_t(V,w,w') \log \frac {2 \, p_t(V,w,w')}{
    \sqrt{h_t(v)h_t(v_*)
      h_t(v')h_t(v_*')} \, \bar\sigma(V;w,w')}  \\
&\qquad\quad  - p_t(V,w,w') + \frac 12 \sqrt{h_t(v)h_t(v_*)
      h_t(v')h_t(v_*')} \, \bar\sigma(V;w,w')
  \Big].
  \end{split}
\end{equation}

\end{lemma}

\begin{proof}
  We first note that
  \begin{equation}
    \label{dirsqrt}
  \mathcal D (\pi) = \frac 12 \int\!M(\de v)M(\de v_*)\,
  r(v,v_*;\de v',\de v_*') \Big[ \sqrt{h(v') h(v'_*)} -\sqrt{h(v) h(v_*)}\Big]^2.     
  \end{equation}
  Indeed, by standard arguments, see e.g.\ \cite[App.~1, Thm.~10.2]{KL},
  $\mc D(\pi)$ is bounded above by the right hand side in the previous
  displayed formula. The converse inequality is obtained by choosing
  as test function a sequence of continuous and bounded $\phi^n$ that
  converges to $\frac 12 \log h$.

  Recalling \eqref{(v)},  the proof of \eqref{Dir1} is achieved by
  expanding the square on the right hand side of \eqref{dirsqrt} and
  using the detailed balance condition. The representation \eqref{Dir2} now follows directly by \eqref{dirsqrt}. Finally, by \eqref{Dir1} and using that $Q(1)<+\infty$, the representation \eqref{Re} is achieved by a direct computation.
\end{proof}

\begin{proof}[Proof of Theorem~\ref{gf}]
  Recalling that $\tilde{\ms S}$ has been introduced in \eqref{def:B},
  we show that \eqref{J=} holds for $(\pi,Q)\in \tilde{\ms S}$.
  We write $\pi_t(dv)= h_t(v) M(\de vv)$,
  $Q(\de t;\de v,\de v_*,\de v',\de v_*')=p_t(V,w,w')\, dt M(\de V) M(\de w) M(\de w')$,
  $ r(v,v_*;\de v',\de v_*') = \bar\sigma(V;w,w') M(\de w')$ where
  $V=(v+v_*)/\sqrt{2}$, $w=(v-v_*)/\sqrt{2}$, and
  $w=(v'-v_*')/\sqrt{2}$.  Setting $\Psi(a,b) := a\log(a/b) -(a-b)$,
  by \eqref{5}
\begin{equation*}
  J(\pi,Q) = \int\!dt M(\de V)M(\de w)M(\de w')\,
  \Psi\big( p_t(V,w,w') , \tfrac 12 h_t(v)h_t(v_*) \bar\sigma(V,w,w') \big).
\end{equation*}
We observe that for each $\bar a>0$,
\begin{equation*}
\Psi (a,b) =   \Psi (\bar a,b)  + \log\frac {\bar a}b   \, (a-\bar a)    
+ \Psi(a,\bar a)
\end{equation*}
and use this decomposition with $\bar a = 
\frac 12 \sqrt{h_t(v)h_t(v_*)h_t(v')h_t(v_*')}
\, \bar\sigma(V;w,w')$. We deduce
\begin{equation*}
   J(\pi,Q) = J_1(\pi,Q) +J_2(\pi,Q) +J_3(\pi,Q)  
\end{equation*}
where, by using the detailed balance condition
$\bar\sigma(V,w,w')=\bar\sigma(V,w',w)$, 
\begin{equation*}
  \begin{aligned}
  J_1(\pi,Q) &= \frac 12 \int\!\de t M(\de V)M(\de w)M(\de w')
  \\
  &\quad \times\Big[\sqrt{h_t(v)h_t(v_*)h_t(v')h_t(v_*')} \bar\sigma(V,w,w')
  \log \frac {\sqrt{h_t(v')h_t(v_*')}} {\sqrt{h_t(v)h_t(v_*)}} 
  \\
  &\quad \; 
  -\sqrt{h_t(v)h_t(v_*)h_t(v')h_t(v_*')}\bar\sigma(V;w,w')
  + h_t(v)h_t(v') \bar\sigma(V;w,w')
  \Big]
\\
&\quad = \frac 12 \int\!\de t M(\de V)M(\de w)M(\de w')
  \\
  &\quad \times \Big[ h_t(v)h_t(v') \bar\sigma(V,w,w') -
\sqrt{h_t(v)h_t(v_*)h_t(v')h_t(v_*')} \bar\sigma(V;w,w')\Big]
\\
 & = \frac 12 \int_0^T\!\de t \,\mc  D(\pi_t).
  \end{aligned}
\end{equation*}
Again by the detained balance condition and the balance equation
\eqref{bal},
\begin{equation*}
  \begin{aligned}
  &J_2(\pi,Q) = \int\!\de t M(\de V)M(\de w)M(\de w')
  \\
  &\quad \times\Big[
  \log \frac {\sqrt{h_t(v')h_t(v_*')}} {\sqrt{h_t(v)h_t(v_*)}}
  \Big(  p_t(V,w,w') - 
  \frac 12 \sqrt{h_t(v)h_t(v_*)h_t(v')h_t(v_*')} \bar\sigma(V;w,w')\Big) \Big]
\\
&= \int\!\de t M(\de V)M(\de w)M(\de w') \, p_t(V,w,w')
\log \frac {\sqrt{h_t(v')h_t(v_*')}} {\sqrt{h_t(v)h_t(v_*)}}
\\
&= \frac12 \Big[ \mc H (\pi_T) - \mc H(\pi_0) \Big]. 
\end{aligned}
\end{equation*}
Finally, 
\begin{equation*}
  \begin{aligned}
  J_3(\pi,Q) &= \int\!\de t M(\de V)M(\de w)M(\de w')
  \\
  &\quad \times\Big[
  p_t(V,w,w')
  \log \frac {2\, p_t(V,w,w')} {
    \sqrt{h_t(v)h_t(v_*)h_t(v')h_t(v_*')}\bar\sigma(V;w,w')}
  \\
  &\quad \; -p_t(V,w,w')
    +\frac 12 \sqrt{h_t(v)h_t(v_*)h_t(v')h_t(v_*')} \bar\sigma(V;w,w')\Big]
\\
& = \frac 12 \mc R(\pi,Q)
  \end{aligned}
\end{equation*}
that concludes the proof of \eqref{J=} when
$(\pi,Q)\in \tilde{S}$.

In particular, we have shown that for $(\pi,Q)\in \tilde{S}$ it holds 
\begin{equation}
  \label{bent}
 \mc H(\pi_T) \le \mc H(\pi_0) +J(\pi,Q)  
\end{equation}
By Theorem~\ref{the:approx} and the lower-semicontinuity of the
relative entropy we then get the previous bound for any path 
$(\pi,Q)\in \hat{S}$; hence $\mc H(\pi_T) <+\infty$ when
$ \mc H(\pi_0)$ and $J(\pi,Q)$ are bounded.

Fix $(\pi,Q)\in \hat{\ms S}$.  To prove that \eqref{J=} holds
for $(\pi,Q)$, we will construct a sequence $\tilde{\ms S}\cap \hat{\ms S}\ni(\pi^n, Q^n)\to (\pi, Q)$ such that
\begin{equation}\label{5ineq}
  \begin{aligned}
     & \varlimsup_{n}
  J(\pi^{n},Q^{n})\leq J(\pi, Q)\\
      &  \varlimsup_{n}
      \mc H(\pi^{n}_0)
\le
\mc H(\pi_0),\qquad   \varlimsup_{n}
      \mc H(\pi^{n}_T)
\le
\mc H(\pi_T),\\
& \varlimsup_{n}
\int_0^T\!\de t \, \mc D(
\pi^{n}_t) \le \int_0^T\!\de t \,\mc D( \pi_t)
\\
& \varlimsup_{n}
\mc R( \pi^{n},
Q^{n}) \le \mc R(\pi,Q).
  \end{aligned}
\end{equation}
Observe in fact that the converse inequalities follows from the
lower-semicontinuity of $J$, $\mc H$, $\mc D$, and $\mc R$.

Let
$(\pi^n, Q^n)$ be the sequence constructed in the proof of Theorem
\ref{the:approx}, so that the first inequality in \eqref{5ineq} holds.  The proof
of the others is achieved in two steps.

\noindent {\bf Step 1 - Convolution.}
Let $(f^{\delta},q^{\delta})$ be the sequence constructed in Step 1 in
the proof of Theorem~\ref{the:approx}, and let $(h_\delta, p_\delta)$ such that $\pi^\delta(\de v)=f^\delta(v)\de v=h^\delta(v)M(\de v)$ and
$dQ^\delta= \de t q_t^\delta(v,v_*,w')\de v \de v_* \de w'= \de t\, p_t^\delta(V,w,w')M(\de V)M(\de w)M(\de w')$.

We observe that proof of the second line in \eqref{5ineq} is  achieved by the same
argument in Step 1 of Theorem~\ref{the:approx}.
We now show that
\begin{equation}\label{D}
\varlimsup_{\delta}\int_0^T \de t\,\mc D(\pi_t^{\delta}) \leq \int_0^T \de t\, \mc D(\pi_t).
\end{equation}
We use the representation of $\mc D$ provided by Lemma \ref{t:rdf}. By \eqref{(v)} and item (ii) in Definition \ref{def:sac} we deduce
\begin{equation*}\begin{split}
    &\lim_{\delta\to 0}\int_0^T \de t \int M(\de v) M(\de v_*) M(\de w') h_t^\delta(v) h_t^\delta(v_*)\sigma(v, v_*, w')\\
    &=
\int_0^T \de t\int  M(\de v) M(\de v_*) M(\de w') h_t(v) h_t(v_*)\sigma(v, v_*, w').
\end{split}\end{equation*}
On the other hand, since $h_t^\delta\to h_t$ a.e.,
by Fatou's lemma
%
%
\begin{equation*}\begin{split}
&\varlimsup_{\delta\to 0}-\int_0^T \de t \int\!M(\de v) M(\de v_*)M(\de w') \, \sqrt{ h_t^{\delta}(v)h_t^{\delta}(v_*)h_t^{\delta}(v')h_t^{\delta}(v_*')}
\, \sigma(v,v_*,w')\\
&\leq  -\int_0^T \de t \int\!M(\de v) M(\de v_*)M(\de w') \, \sqrt{ h_t(v)h_t(v_*)h_t(v')h_t(v_*')}
\, \sigma(v,v_*,w'),
\end{split}\end{equation*}
which concludes the proof of \eqref{D}. Observe that the lower semicontinuity of $\mc D$ actually implies $\int_0^T \de t \, \mc D(\pi_t^\delta)\to \int_0^T \de t \,\mc D(\pi_t)$, so that by Lemma \ref{t:rdf}
\begin{equation}\begin{split}\label{D1}
&\lim_{\delta\to 0}\int_0^T \de t \int\!M(\de v) M(\de v_*)M(\de w') \, \sqrt{ h_t^{\delta}(v)h_t^{\delta}(v_*)h_t^{\delta}(v')h_t^{\delta}(v_*')}
\, \sigma(v,v_*,w')\\
&=\int_0^T \de t \int\!M(\de v) M(\de v_*)M(\de w') \, \sqrt{ h_t(v)h_t(v_*)h_t(v')h_t(v_*')}
\, \sigma(v,v_*,w').
\end{split}\end{equation}
We conclude the step by showing that
\begin{equation}\label{R}
\varlimsup_{\delta\to 0} \mc R(\pi^\delta, Q^\delta) \leq \mc R(\pi, Q).
\end{equation}  
By the representation provided by \eqref{Re} and Lemma \ref{l:qlog}
\begin{equation*}\begin{split}
    \mc R(\pi,Q)= &Q(\Phi)+ Q(\Phi')-2Q(\log\sigma)-2Q(1)\\
    &+\int \de t\,M(\de v)M(\de v_*)M(\de w')\sqrt{h_t(v)h_t(v_*)h_t(v')h_t(v'_*)}\sigma(v,v_*,w'),
\end{split}\end{equation*}
where $\Phi=\log \frac {2q(v,v_*,w')} {f(v) f(v_*) g_1(w')}$ ,
$\Phi'=\log \frac {2q(v,v_*,w')} {f(v') f(v'_*) g_1(w)}$ , 

We start by observing that $Q^\delta(1)\to Q(1)$, then in view of  \eqref{D1} it is enough to show that
\begin{equation*}\begin{split}
& \varlimsup_{\delta\to 0}\int \! \de t M(\de V) M(\de w) M(\de w')\,
  p^\delta_t(V,w,w') \log \frac {2\, p^{\delta}_t(V,w,w')}{
    \sqrt{h^{\delta}_t(v)h^{\delta}_t(v_*)
      h^{\delta}_t(v')h^{\delta}_t(v_*')} \, \bar\sigma(V;w,w')}\\
    &\leq \int \! \de t M(\de V) M(\de w) M(\de w')\,
  p_t(V,w,w') \log \frac {2\, p_t(V,w,w')}{
    \sqrt{h_t(v)h_t(v_*)
      h_t(v')h_t(v_*')} \, \bar\sigma(V;w,w')}.
    \end{split}\end{equation*}  
Observe that
\begin{equation*}\begin{split}
& \int \! \de t M(\de V) M(\de w) M(\de w')\,
  p^\delta_t(V,w,w') \log \frac {2\, p^{\delta}_t(V,w,w')}{
    \sqrt{h^{\delta}_t(v)h^{\delta}_t(v_*)
      h^{\delta}_t(v')h^{\delta}_t(v_*')} \, \bar\sigma(V;w,w')}\\
  & = \frac 1 2 \int \! \de t \de v  \de v_* \de w' \,q_t^\delta(v, v_*, w') \log
  \frac {2\, q_t^\delta(v,v_*,w')}{f_t^\delta(v)f_t^\delta(v_*)g_{1+\delta}(w')}\,\\
  & + \frac 1 2 \int \! \de t \de v \de v_* \de w' \,q_t^\delta(v, v_*, w') \log
  \frac {2 \, q_t^\delta(v,v_*,w')}{f_t^\delta(v')f_t^\delta(v'_*)g_{1+\delta}(w)}\,\\
 & +\frac 1 2 \int \! \de t \de v \de v_* \de w' \,q_t^\delta (v, v_*, w')\log \frac {g_{1+\delta}(w)g_{1+\delta}(w')}{B(v,v_*,w')^2}.
    \end{split}\end{equation*}  
  The proof is achieved by the same argument in Step 1 of Theorem
  \ref{the:approx}.

  \medskip
  \noindent {\bf Step 2 - Truncation of the flux.}
  As in the Theorem~\ref{the:approx} we now assume $f$ strictly
  positive on compacts uniformly in time and 
  $q_t^{(3)}\in L^1([0,T]; L^2(\bb R^d))$. We denote by 
  $(f^{\ell},q^{\ell})$ the sequence constructed in Step 2 in
  the proof of Theorem~\ref{the:approx}.

  By the argument in Step 2 of Theorem~\ref{the:approx},
  $\lim_{\ell} \mc H(\pi^\ell_0) =\mc H(\pi_0)$, we now show that 
  \begin{equation}
    \label{H2}
    \varlimsup_{\ell} \mc H(\pi^\ell_T) \le \mc H(\pi_T).
  \end{equation}
  Recalling \eqref{tildeell} we write
  \begin{equation*}
    \begin{split}
          f^\ell_T &= c_\ell f_0 +
    2c_\ell \left( \int_0^T \de s
          \big[ q^{(3)}_s-
      \tilde q^{\ell,(1)}_s \big]
    \right)  \id_{|v|\le \ell}
    \\
    & =
    \begin{cases}
     \displaystyle{  c_\ell f_T + 2c_\ell \int_0^T \de s
          \big[ q_s^{(1)} - \tilde q^{\ell,(1)}_s \big]}
          & \textrm{if $|v|\le \ell$} \\
      c_\ell f_0   & \textrm{if $|v|> \ell$}.  
    \end{cases}      
    \end{split}
  \end{equation*}
  Set 
  \begin{equation*}
    \hat{f}^\ell_T :=
    c_\ell a_\ell \Big\{ f_T + 2 \Big( \int_0^T\!\de s\,
    \big[ q_s^{(1)} - \tilde q^{\ell,(1)}_s \big] \big) \id_{|v|\le
      \ell} \Big\}
  \end{equation*}
  where $1/a_\ell = c_\ell + 1/c_\ell - 1$. Observe that
  $\hat{f}^\ell_T$ is a probability density. Since
  $0< c_\ell a_\ell < 1$, by using \eqref{bent} and applying the
  argument leading to \eqref{dario3}, we deduce that
  $\varlimsup_{\ell} \mc H(\hat{f}^\ell_T )\le \mc H({f}_T )$
  and 
  $\lim_{\ell} \big| \mc H({f}^\ell_T ) - \mc H(\hat {f}^\ell_T )
  \big| = 0$. This completes the proof of \eqref{H2}.

By  the representation of $\mc D$ provided by Lemma \ref{t:rdf}, using  \eqref{ui} and Fatou's lemma we conclude that 
  \begin{equation}\label{D2}
    \varlimsup_{\ell}\int_0^T \de t\,\mc D(\pi_t^{\ell})
    \leq \int_0^T \de t\, \mc D(\pi_t).
  \end{equation}
  
  It remains to  show that
  \begin{equation}\label{R2}
\varlimsup_{\ell} \mc R(\pi^\ell, Q^\ell)\leq \mc R(\pi, Q).
  \end{equation}  
  This  is  achieved by using  the representation \eqref{Re} and the argument in Step 2 of Theorem \ref{the:approx}.

  To prove the second statement of the theorem, observe that if  $I(\pi, Q)=0$ then $\pi$ has bounded second moment, $Q=Q^\pi$, and $H(\pi\vert m)=0$.
  Therefore  $\pi_0=m$ and, when  $\lambda$ is bounded, $(\pi, Q^\pi)\in \hat{\ms S}$. By \eqref{J=}, the inequality \eqref{epi} amounts to $J(\pi, Q^\pi)\leq 0$.
  \end{proof}

\section*{Acknowledgments}
We are grateful to R. Di Leonardo for suggesting us the example of molecular gases. We also thank M. Erbar for useful discussions.

\end{document}